\documentclass[oneside,a4paper]{amsart} 
\usepackage{amssymb,verbatim,mathabx,MnSymbol}

\usepackage[shortlabels]{enumitem}
    \setenumerate[0]{label={\rm (\roman*)}, leftmargin=*}

\usepackage[T1]{fontenc}
\usepackage[english]{babel}

\usepackage[textsize=footnotesize,textwidth=20ex,colorinlistoftodos]{todonotes}

\usepackage{hyperref}
\hypersetup{
    colorlinks=true,
    linkcolor=blue,
    filecolor=magenta, 
    citecolor=red,     
    urlcolor=cyan,
}

\usepackage[capitalize]{cleveref}
    \crefname{enumi}{}{}
    \Crefname{enumi}{Item}{Items}
    \crefname{equation}{}{}
    \Crefname{equation}{Equation}{Equations}

\newtheorem{proposition}{Proposition}[section]
\newtheorem{lemma}[proposition]{Lemma}
\newtheorem{corollary}[proposition]{Corollary}
\newtheorem{theorem}[proposition]{Theorem}

\theoremstyle{definition}

\newtheorem{example}[proposition]{Example}

\newtheorem{remark}[proposition]{Remark}

\DeclareMathOperator{\Aut}{\mathrm{Aut}}

\begin{document}
\title{Graphs, Axial Algebras and their Automorphism Groups}
\author{Hans Cuypers} 

\begin{abstract}
We introduce a class of  algebras over a field $\mathbb{F}$ related to directed graphs in which all edges are labeled by nonzero elements of the field $\mathbb{F}$. If all labels are different from $1$, these algebras are axial algebras. 
We determine their fusion laws, prove them to be simple in almost all cases, and determine their automorphism group under some conditions on the degrees and girth of the graph.

A construction of a class of these graphs with prescribed automorphism group enables us to construct  
for each group $G$  infinitely many simple (axial) algebras (with a fixed fusion law) such that the automorphism group of the algebra is isomorphic to $G$.
\end{abstract}

\maketitle

\bigskip

\emph{To Arjeh, Jonathan and Sergey, on the occasion of their 75th and 65th birthday}


\bigskip

\section{Introduction}

Axial algebras form a  class of not necessarily associative algebras.
These algebras are defined as being generated by a class of \emph{idempotents}, called \emph{axes},
for which the left and right adjoint maps are semi-simple and obey so-called \emph{fusion laws}.
These laws describe how the product of two eigenvectors of the adjoint maps of an axis
decomposes into eigenvectors.

More precisely,  let $A$ be an algebra over a field $\mathbb{F}$, with $A$ possibly non-commutative or non-associative, and $x\in A$ an element such that
$x$ is an idempotent, i.e. $x^2=x$, and the \emph{left adjoint map} ${L}_x:A\rightarrow A$ defined by $L_x(a)=xa$ for all $a\in A$ 
and \emph{right adjoint map} $R_x:A\rightarrow A$ defined by $R_x(a)=ax$ for all $a\in A$ are 
semi-simple. So $$A=\bigoplus_{\lambda\in \mathrm{spec}(L_x)} A_\lambda(L_x)$$
and 
$$A=\bigoplus_{\lambda\in \mathrm{spec}(R_x)}A_\lambda(R_x),$$
where $\mathrm{spec}(L_x)$ and $\mathrm{spec}(R_x)$ are  the \emph{spectrum} (set of eigenvalues) of $L_x$ and $R_x$, respectively,  and, for $\lambda$ in $\mathrm{spec}(L_x)$ or $\mathrm{spec}(R_x)$, the subspaces $A_\lambda(L_x )=\{a\in A\mid L_x(a)=\lambda a\}$ and $A_\lambda(R_x )=\{a\in A\mid R_x(a)=\lambda a\}$ are 
the corresponding eigenspaces.
We call such $x$ an \emph{axis} of $A$.

A \emph{fusion law} is a pair $(\mathcal{F},\star)$ consisting of a subset $\mathcal{F}$ of $\mathbb{F}$
and a map $\star$ from $\mathcal{F}\times \mathcal{F}$ to the power set $2^\mathcal{F}$ of $\mathcal{F}$
(with infix notation).

A pair $(A,X)$ consisting of an  algebra $A$ over a field $\mathbb{F}$ with generating set $X\subseteq A$ of \emph{axes}, and  left and right \emph{fusion laws} 
$(\mathcal{F}_L,\star_L)$ and $(\mathcal{F}_R,\star_R)$, respectively, 
is called an \emph{axial algebra satisfying the left fusion law $(\mathcal{F}_L,\star_L)$ and right fusion law $(\mathcal{F}_R,\star_R)$} 
if  for all $x\in X$ we have:

\begin{enumerate}
\item The element $x$ is an axis with $\mathrm{spec}(L_x)\subseteq \mathcal{F}_L$ and $\mathrm{spec}(R_x)\subseteq \mathcal{F}_R$.
\item For all $\lambda,\mu\in \mathcal{F}_L$ we have  
$$A_\lambda(L_x) A_\mu(L_x)\subseteq \bigoplus_{\rho\in \lambda\star_L\mu}A_\rho(L_x)$$
and for all $\lambda,\mu\in \mathcal{F}_R$ we have $$A_\lambda(R_x) A_\mu(R_x)\subseteq \bigoplus_{\rho\in \lambda\star_R\mu}A_\rho(R_x).$$
\end{enumerate}

We also call $A$ an axial algebra if the set $X$ of axes is clear from the context.
If $(\mathcal{F}_L,\star_L)$ and $(\mathcal{F}_R,\star_R)$ are equal, we say the algebra $A$ with generating set of axes $X$ satisfies the fusion law 
$(\mathcal{F}_L,\star_L)=(\mathcal{F}_R,\star_R)$.
It is clearly  possible to choose these fusion laws to be equal when $A$ is a commutative axial algebra generated by axes in $X$.

An axis $x\in X$ is called \emph{primitive} if the multiplicity of the eigenvalue $1$ for both $L_x$ and $R_x$ is equal to one, so $A_1(L_x)=A_1(R_x)=\langle x\rangle$.

In general axial algebras will satisfy quite wild fusion laws, and it seems hard 
to develop a general theory. However, the study of axial algebras satisfying strict fusion laws with small sets $\mathcal{F}_L$ and $\mathcal{F}_R$ is more promising.

If the algebra $A$ is commutative, then for each axis $a$ we have that $L_a=R_a$ and we can identify the two maps $\star_L=\star_R$

Of particular interest are commutative axial algebras satisfying fusion laws of \emph{Jordan type} $\mathcal{J}(\eta)$ 
or of \emph{Monster type} $\mathcal{M}(\alpha,\beta)$.
These laws are given in \cref{jordan-monster-type}. (In the table, rows and columns are labeled by the elements of $\mathcal{F}$.
For $\lambda, \mu\in \mathcal{F}$ we find $\lambda\star \mu$ to be the set of elements given in the $\lambda,\mu$ entry of the table.)
The  interest in (commutative) axial algebras satisfying these particular laws  comes from examples related  to (the three sporadic) Fischer groups and to the Monster.
\begin{table}
$$\begin{array}{|l||l|l|l|}
\hline
\star & 1 &  0 &\eta\\
\hline\hline
1& 1&  & \eta\\
\hline
0&  & 1,0 & \eta\\
\hline
\eta &\eta& \eta  & 1,0 \\
\hline
\end{array}
\hspace{1cm}\begin{array}{|l||l|l|l|l|}
\hline
\star & 1 & 0& \alpha &\beta\\
\hline\hline
1& 1& &\alpha & \beta\\
\hline
0    &  & 0& \alpha & \beta\\
\hline
\alpha& \alpha &\alpha&  1,0& \beta\\
\hline
\beta &\beta& \beta &\beta& 1,0,\alpha \\
\hline
\end{array}
$$ 
\caption{Fusion laws of type $\mathcal{J}(\eta)$ 
and  $\mathcal{M}(\alpha,\beta)$}
\label{jordan-monster-type}
\end{table}

Indeed, (commutative) axial algebras have been introduced by Hall, Shpectorov and Rehren in \cite{hrs} where the focus is on axial algebras of Jordan type $\mathcal{J}(\eta)$ generated by a set of primitive axes. Our definition of axial algebra generalizes this concept to possibly non-commutative algebras.
Hall,  Shpectorov and Rehren  have been able to  prove that these commutative axial algebras are (quotients of) Matsuo algebras related to $3$-transposition groups, including the three sporadic Fischer groups, or $\eta=\frac{1}{2}$. (See \cite{hrs}.)
In the latter case, the classification is still open. Not only Matsuo algebras occur, but also those Jordan algebras which are generated by primitive idempotents, as their axial decomposition then equals the Peirce decomposition, which is of type  $\mathcal{J}(\frac{1}{2})$.
Among the axial algebras of Monster type we find the Griess algebras related to the Monster.

Meanwhile, several other axial algebras have been studied, related to $3$-trans\-po\-sition groups \cite{hrs,joshi,alsaeedi}, Majorana algebras \cite{majorana}, Chevalley groups \cite{demedts}, codes \cite{codealgebras}, Steiner systems \cite{fox}, or geometries \cite{hans}.  
Non-commutative axial algebras are considered  by Rowen and Segev in \cite{segev1,segev2,segev3}.

In this paper we study a class of (axial) algebras defined by edge labeled directed graphs, which have no loops nor multiple edges.
The \emph{underlying graph} $\underline\Gamma$  of a directed graph $\Gamma$ is the (ordinary) graph on the vertices of $\Gamma$ and whose edges are the sets $\{x,y\}$
where $(x,y)$ is an edge of $\Gamma$. The graph $\Gamma$ is called \emph{weakly connected} if and only if its underlying graph is connected.

Throughout this paper, $\mathbb{F}$ denotes a field and  $\Gamma=(X,E)$ a directed graph 
on the vertex set $X$ with directed edge set $E$, but no loops nor multiple edges.
The edges in $E$ are labeled by nonzero elements from $\mathbb{F}$.
The label of an edge $(x,y)$ is denoted by  $\alpha_{x,y}$. 

The algebra $A:=A_{\Gamma}$
is then  the vector space $\mathbb{F}X$ with basis $X$ and bilinear product for all $x,y\in X$ defined by
$$xy=\begin{cases}
x&\mathrm{if}\ x=y,\\
\alpha_{x,y} (x+y)&\mathrm{if}\ (x, y) \ \mathrm{is\ an\ edge\ },\\
0&\mathrm{if}\ (x,y)\ \mathrm{is \ not \ an \ edge}.\\
\end{cases}$$

The algebra $A_\Gamma$ is commutative if and only if for all edges $(x,y)\in E$ also $(y,x)\in E$ and $\alpha_{x,y}=\alpha_{y,x}$.

If $\Gamma^\top$ is the reversed graph of $\Gamma$, i.e. the graph with the same vertex set and  edges $(y,x)$, where $(x,y)\in E$,
then the algebra $A_{\Gamma^\top}$ in which an edge $(x,y)$ of $\Gamma^\top$ has label $\alpha_{y,x}$ is the opposite  algebra of the algebra $A_\Gamma$.

In our first result we determine the structure of such algebras, and show them to be axial algebras if none of the labels is $1$.

\begin{theorem}\label{algebrathm}
Let $\mathcal{F}$ be a subset of the field $\mathbb{F}$ containing $1$.
Let $\Gamma=(X,E)$ be a weakly connected directed graph with edges labeled  by elements in $\mathcal{F}$ different from $0$.
Then we have the following:

\begin{enumerate}
\item
The algebra $A_\Gamma$ over $\mathbb{F}$ is simple, unless one of the following occurs:
\begin{enumerate}[\rm(a)]
\item $\Gamma$ contains a complete subgraph $\Delta=(Y,F)$ with $|Y|>1$ and all labels equal to $\frac{1}{2}$.
If $x$ is a vertex not in $\Delta$ and there is an edge $(x,y_0)$ or $(y_0,x)$ in $E$ for some vertex $y_0$ of $\Delta$, then for all
vertices $y$ of $\Delta$ we have $(x,y)$ as well as $(y,x)$ in $E$, and,  moreover, $\alpha_{x,y}=\alpha_{x,y_0}$ and $\alpha_{y,x}=\alpha_{y_0,x}$.

In this case $A_\Gamma$ has the following ideal
$\{\sum_{y\in Y} \lambda_yy\in A\mid \sum_{y\in Y}\lambda_y=0\}$.
\item $\Gamma$ is a finite complete graph in which all labels are equal to $\frac{1}{2-|X|}$; in this case $\langle \sum_{x\in X} x\rangle$ is an ideal.
\end{enumerate}
\item
The algebra $A_{\Gamma}$ is an axial algebra generated by the axes in $X$ satisfying left and right fusion laws of type $\mathcal{G}(\mathcal{F})$ as in \cref{fusionlaw},
provided none of the labels equals $1$ and $0\in \mathcal{F}$ if $\Gamma$ is not complete.
\end{enumerate}
\begin{table}[h]
$$\begin{array}{|l||l|l|l|l|l|}
\hline
\star & 1 &  \alpha&\dots &\beta &(0)\\
\hline\hline
1& 1&\alpha& & \beta&0\\
\hline
\alpha& \alpha& 1,\alpha & &1,\alpha,\beta&1,\alpha,0\\
\hline
\vdots  &&&&&\phantom{a}\\
\hline
\beta &\beta& 1,\alpha,\beta&  & 1,\beta& 1,\beta,0 \\
\hline
(0)     &   0 & 1,\alpha,0&& 1,\beta,0 &0\\
\hline
\end{array}$$

\caption{The fusion law of Graph type $\mathcal{G}(\mathcal{F})$, where $\mathcal{F}=\{1,\alpha,\dots,\beta,(0)\}$.
The element $0$ is between brackets as it may not be an eigenvalue. This is the case when any two vertices are adjacent.}\label{fusionlaw}
\end{table}
\end{theorem}

The above theorem implies that even in the case that $\mathcal{F}=\{0,1,\alpha\}$ for some fixed $\alpha\in \mathbb{F}\setminus \{0,1\}$
there is a wild variety of axial algebras of graph type $\mathcal{G}(\mathcal{F})$. 

It is immediate from the definition of the algebra $A=A_{\Gamma}$ that the automorphism group
of the labeled graph $\Gamma$ embeds injectively into the automorphism group $\mathrm{Aut}(A)$ of $A$.
The group $\mathrm{Aut}(\Gamma)$ need not be the full automorphism group of the algebra $A$.

However, we are able to prove the following two results which are concerned with \emph{symmetric} directed graphs $\Gamma=(X,E)$ (which
means that if $(a,b)\in E$, then also $(b,a)\in E$) providing conditions which imply $\mathrm{Aut}(\Gamma)$ and $\mathrm{Aut}(A)$ to be isomorphic.

Let $\Gamma=(X,E)$ be a symmetric directed graph.
Then we define the \emph{degree} $\delta_v$ of a vertex $v$ to be equal to the indegree as well as to the outdegree of $v$, which are equal for symmetric directed graphs.
So $\delta_v=|\{w\in X\mid (v,w) \in E\}|$.
Moreover, by $k_{\mathrm{min}}$ and $k_{\mathrm{max}}$ we denote the minimum and maximum of the degrees $\delta_v$, with $v\in X$, if they exist.
Otherwise they are set to be $\infty$.

The \emph{girth} of $\Gamma$ (and of $\underline{\Gamma}$) is the length of the smallest cycle in $\underline{\Gamma}$, if such cycle exists. Otherwise it is set to be $\infty$.

\begin{theorem}\label{graphthm}
Suppose $\mathcal{F}$ is a subset of  the field $\mathbb{F}$.
Let $\Gamma$ be an edge labeled weakly connected symmetric directed  graph with labels in $\mathcal{F}\setminus \{0,1\}$.
Let $A_\Gamma$ be the corresponding $\mathbb{F}$-algebra.
Suppose $\Gamma$ has finite girth $g$.

If both  $2<k_\mathrm{min}<g-2$ and   $k_\mathrm{max}\leq \min(2(k_{\mathrm{min}}-1),g-3)$,
then $\mathrm{Aut}(A_{\Gamma})$ is isomorphic to  $\mathrm{Aut}(\Gamma)$, the automorphism group of the labeled graph $\Gamma$.
\end{theorem}

In the smallest case where $k_\mathrm{min}=k_\mathrm{max}=3$ and $g\geq 6$ we already encounter various  interesting graphs
satisfying the hypothesis of the theorem. 
Indeed, let $\Pi$ be a generalized $n$-gon with $n\geq 3$, with three points per line and three lines per point
and $\Gamma$ its directed incidence graph, i.e. the graph with as vertices the points and lines of $\Pi$ and as edges  pairs $(x,y)$ with $x$ a point on a line $y$ or $x$ a line through the point $y$. Then $\Gamma$ has valency $k_\mathrm{min}=k_\mathrm{max}=3$ and girth $g=2n\geq 6$.
For each labeling of the edges $(x,y)$, where $x$ is a point, with  an element $\alpha\in \mathbb{F}\setminus\{0,1\}$ and each edge $(x,y)$ with $x$ a line
with $\beta\in \mathbb{F}\setminus\{0,1\}$  we obtain an algebra $A_\Gamma$ whose automorphism group is isomorphic to
the automorphism group of $\Gamma$.  
If $n=3$ and $\alpha\neq \beta$ this automorphism group is isomorphic to $\mathrm{PSL}_3(2)$, while if $\alpha=\beta$,
the automorphism group is $\mathrm{PSL}_3(2):2$.
For $n=4$ and $\alpha\neq \beta$ we find the automorphism group to be isomorphic to $\mathrm{Sym}_6$ and for $n=6$ to $\mathrm{G}_2(2)$.
In case $\alpha=\beta$
the groups are twice as big.
Notice the algebra $A_\Gamma$ is commutative if and only if $\alpha=\beta$.

In the above theorem we insist on information on all valencies of vertices and relate them to the girth of the graph. By considering
the directed incidence graph $\Gamma$ of a graph or partial linear space $\Delta$,  we can say even more.
This is the graph with vertex set consisting of the  vertices and edges of $\Delta$ and as edges the pairs $(x,e)$ and $(e,x)$ where $x$ is a vertex on the edge or line $y$ of  $\Delta$.

\begin{theorem}\label{incgraphthm}
Suppose $\mathcal{F}$ is a subset of  the field $\mathbb{F}$.
Let $\Delta$ be a connected graph with all vertices  of degree at least three or a connected partial linear space with three points per line and all points on at least four lines.
Let $\Gamma$ be the directed incidence graph of $\Delta$.
Suppose the edges of $\Gamma$ are labeled with elements from  $\mathcal{F}\setminus\{0,1\}$ or in case $\mathbb{F}=\mathbb{F}_2$ with $1$.

Then $\mathrm{Aut}(A_{\Gamma})$ is isomorphic to  $\mathrm{Aut}(\Gamma)$, the automorphism group of the labeled graph $\Gamma$.
\end{theorem}

As the restrictions on $\Delta$ are rather weak, we can apply our results to many interesting graphs or partial linear spaces $\Delta$.
We refer the reader to \cref{section:auto} for such examples.

We notice that the graphs in the above theorems need not be finite graphs, but in \cref{graphthm}, the vertices of these graphs do of course have finite degree.

Popov \cite{popov} raised the question whether each finite group is the automorphism group of a finite dimensional
simple algebra. He and Gordeev \cite{gordeev} provided an affirmative answer for sufficiently large  fields.

In \cite{evolution} Costoya, Mu\~{n}oz, Tocino, and Viruel showed that for each finite group $G$ there is a finite dimensional simple evolution algebra
whose automorphism group is isomorphic to $G$. They require the field $\mathbb{F}$ to be of size bigger than twice the order of the group.
The same result has been obtained in \cite{evolution2} for fields of characteristic $0$.

Using Frucht's famous theorem \cite{frucht39,frucht49} that   each finite group $G$ is the automorphism group of a finite graph,
and its extension to infinite groups and graphs by de Groot \cite{groot} and Sabidussi \cite{sabidussi2}, we can show:

\begin{theorem}\label{groupthm}
Let $G$ be a  group and $\mathbb{F}$ a field.
\begin{enumerate}
\item If $\mathbb{F}$ contains at least three elements, then
there exist infinitely many non-isomorphic  directed graphs $\Gamma=(X,E)$ with edges labeled by elements of $\mathbb{F}\setminus\{0,1\}$
such that $A_\Gamma$ is a commutative simple axial  algebra over $\mathbb{F}$ generated by a set of primitive idempotents $X$
of graph type $\mathcal{G}(\mathcal{F})$ with $\mathcal{F}$ being the set containing $0,1$ and all labels, and  $G$ isomorphic to $\mathrm{Aut}(A_\Gamma)$.
\item
If $\mathbb{F}$ contains at least four elements, then there exist infinitely many non-isomorphic  directed graphs $\Gamma=(X,E)$ with edges labeled by elements of $\mathbb{F}\setminus \{0,1\}$
such that $A_\Gamma$ is a  non-commutative simple  axial algebra over $\mathbb{F}$ generated by a set of primitive idempotents $X$
of graph type $\mathcal{G}(\mathcal{F})$, where $\mathcal{F}$ is the set containing $0,1$ and all labels, and  with $G$ isomorphic to $\mathrm{Aut}(A_\Gamma)$.
\item
If $\mathbb{F}$ has only two elements, then there exists infinitely many non-isomorphic directed graphs $\Gamma$ with edges labeled by $1$ such that $A_\Gamma$ over $\mathbb{F}_2$ is simple and has automorphism group isomorphic to $G$.
\end{enumerate}
Moreover, in case $G$ is a finite group, the graphs $\Gamma$ in the above statements can be assumed to be finite, implying the  algebras $A_\Gamma$ to be finite dimensional. 
\end{theorem}

Gorshkov, M$^\mathrm{c}$Inroy, Shumba and Shpectorov showed \cite{gorshkov2023automorphism} that if $\frac{1}{2}\not\in \mathcal{F}$,
then the automorphism group of a finite dimensional commutative axial algebra with  fusion law defined on $\mathcal{F}$ will be finite.
The above \cref{groupthm} shows  that each finite group is the automorphism group of a finite dimensional commutative axial algebra,
providing a counterpart to this result.

The paper is organised as follows.
In the next section we consider for a weakly connected and labeled directed graph $\Gamma$ the algebra $A_\Gamma$ and find all possible ideals, leading to a proof of the first part of \cref{algebrathm}.
In \cref{section:axial} we determine the fusion laws found in \cref{algebrathm}.

In \cref{section:idempotent} we investigate idempotents $a$ in an algebra $A_\Gamma$ for a labeled graph $\Gamma=(X,E)$ for which ${L}_a$ and ${R}_a$ have a small rank as linear map from $A$ to itself.
In particular, we deduce criteria on this rank in relation to the girth of the graph $\Gamma$ and the degrees  of its vertices to be able to conclude that the element $a$ is a member of $X$.  
These investigations form the heart of the proofs of \cref{graphthm} and \cref{incgraphthm}, which are provided in \cref{section:auto}.

In the final \cref{section:frucht} we deduce  \cref{groupthm}  from \cref{incgraphthm}.

\medskip
\noindent
{\bf Acknowledgment}. The author wants to thank Yoav Segev for several constructive comments on a previous version of this paper.
Moreover, Arjeh Cohen and Jonathan Hall celebrated their $75^{th}$ birthday in 2024. Sergey Shpectorov celebrated his $65^{th}$ birthday.
Their mathematics and friendship has been an inspiration for my work in general and this paper in particular.
Finally, we thank the anonymous referee for several comments improving the presentation of this paper. 

\section{Algebras from graphs}
\label{section:algebra}
Let $\mathbb{F}$ be a field, $\mathcal{F}\subseteq \mathbb{F}$ be a set  with $1\in \mathcal{F}$
and $\Gamma=(X,E)$  a weakly connected directed graph, with vertex set $X$ and edge set $E$ and all edges  labeled
with  elements from the set $\mathcal{F}$ different from $0$. (The condition that $1\in \mathcal{F}$ will play a role in the following sections.)

By $A:=A_{\Gamma}$ we denote the algebra on the vector space $\mathbb{F}X$ with basis $X$ and bilinear product defined by 
$$xy=\begin{cases}
x&\mathrm{if}\ x=y,\\
\alpha_{x,y} (x+y)&\mathrm{if}\ (x,y) \ \mathrm{is\ an\ edge\ with\ label \ }\alpha_{x,y},\\
0 &\mathrm{if}\ (x,y)\ \mathrm{is\ not\ an\ edge},\\
\end{cases}$$
for all $x,y\in X$.

If $a\in A$, then, as $X$ is a basis of $A$, we find  $a=\displaystyle\displaystyle\sum_{x\in X}\lambda_x x$, where  $\lambda_x\in \mathbb{F}$.
By $\mathrm{supp}(a)$ we denote the set of elements $x\in X$ with $\lambda_x\neq 0$. It is called the \emph{support} of $a$.
Notice that this support is a finite subset of $X$.

For $x,y\in X$ we write $x\sim y$ and call $y$ \emph{adjacent} to (or a \emph{neighbor} of)  $x$ if and only if $x$ and $y$ are distinct and  $(x,y)$  is an edge, by $x\not\sim y$ we denote that $x$ and $y$ are distinct and not adjacent.

\begin{example}\label{cayley}
Let $G$ be a group and $S$ a set of generators for $G$.
Consider $\Gamma$ to be the \emph{Cayley graph} of $G$ with respect to the generators in $S$.
This is the labeled directed graph with vertex set $X=G$ and  edges the pairs $(g,gs)$ with label $s$, where $g\in G$ and $s\in S$.

Suppose $\mathbb{F}$ is a field and  $\alpha:S\rightarrow\mathbb{F}\setminus \{0\}$  a map.
Then for all $g\in G$ and $s\in S$ we can replace the label $s$ of an edge of $\Gamma$ with $\alpha(s)$.
The algebra $A_\Gamma$ of this so obtained labeled graph $\Gamma$ admits the group $G$ as a group of automorphisms acting by left multiplication on the basis elements in $X$. 
\end{example}

\begin{lemma}
If $\Gamma$ contains an edge, then $A$ is non-associative.
\end{lemma}

\begin{proof}
Suppose $(x,y)$ is an edge in $\Gamma$.
Then in $A_\Gamma$ we find $$(xx)y=xy=\alpha_{x,y} (x+y),$$
while $$x(xy)=x(\alpha_{x,y}(x+y))=\alpha_{x,y} x+ \alpha_{x,y}^2(x+y).$$
So, $A$ being associative implies 
$\alpha_{x,y}=\alpha_{x,y}+\alpha_{x,y}^2$ and hence
$\alpha_{x,y}=0$, which contradicts our assumptions.
\end{proof}

We will determine the structure of the algebra $A=A_\Gamma$.
Before doing so, we first give a few properties of the multiplication of an arbitrary element $a\in A$ with elements from $X$,
which we will use without reference in the sequel of this paper.

\begin{lemma}
Let $a$ be an element of $A$ and $x\in X$.
Then we have the following:
\begin{enumerate}
\item   $\mathrm{supp}(ax)$  equals $\{z\in \mathrm{supp}(a)\mid z\sim x\}$
or $\{z\in \mathrm{supp}(a)\mid z\sim x\}\cup \{x\}$
and $\mathrm{supp}(xa)$  equals $\{z\in \mathrm{supp}(a)\mid x\sim z\}$
or $\{z\in \mathrm{supp}(a)\mid x\sim z\}\cup \{x\}$.

\item If $x\not \in \mathrm{supp}(a)$ is on  an edge $(x,z)$ for only one element $z$ in  $\mathrm{supp}(a)$, then $\mathrm{supp}(xa)=\{x,z\}$.
If $x$ is on an edge $(z,x)$ for only one element $z$ in  $\mathrm{supp}(a)$, then $\mathrm{supp}(ax)=\{x,z\}$.
\end{enumerate}
\end{lemma}

\begin{proof}
This follows immediately from the definition of the product.
\end{proof}

An induced subgraph $\Delta=(Y,F)$ of $\Gamma$ (with labels inherited from $\Gamma$) is called an \emph{ideal subgraph}
if $|Y|\geq 2$,  all vertices in $Y$ are adjacent to each other,   all edges of $\Delta$ are labeled by $\frac{1}{2}$, and, moreover, for each vertex $x \in X\setminus Y$ the vertex $x$
is adjacent to no vertex in $Y$ or to all, and every vertex $y\in Y$ is adjacent to $x$ or none is.
Finally, for fixed $x\in X\setminus Y$ all  edges
$(x,y)$  with $y\in Y$  have the same label $\alpha_x$, while the edges $(y,x)$ all have the same label $\beta_x$.

Given an ideal subgraph $\Delta=(Y,F)$ of $\Gamma=(X,E)$ we define the graph $\Gamma/Y$ 
to be the graph with vertex set $(X\setminus Y)\cup \{\Delta\}$ and edges $(x,y)$ where $x$ and $y$ are vertices from $X\setminus Y$ adjacent in $\Gamma$ with label $\alpha_{x,y}$ and
$(x,\Delta)$ where $x\in X\setminus Y$ is adjacent  to a vertex $y\in Y$ with label $\alpha_{x}$
and  $(\Delta,x)$ with label $\beta_{x}$.
We contract the subgraph $\Delta$ to a single point.

\begin{lemma}\label{idealgraph}
Suppose  $\Delta=(Y,F)$ is an ideal subgraph of $\Gamma$. Then any subspace of $$I_Y=\langle\displaystyle\sum_{y\in Y} \lambda_y y\in A\mid\displaystyle\sum_{y\in Y} \lambda_y=0 \rangle$$
is a two-sided ideal of $A$.

The algebra $A/I_Y$ is isomorphic to $A_{\Gamma/Y}$. 
\end{lemma}

\begin{proof}
Fix an ideal subgraph $\Delta=(Y,F)$ and let $z\in I_Y$. So 
$z=\displaystyle\sum_{y\in Y} \lambda_y y$ such that $\displaystyle\sum_{y\in Y} \lambda_y=0$.

Now consider an element  $x\in X$. 
If $x\in Y$,  then
$$\begin{array}{ll}
xz
&=\lambda_x x+{\displaystyle\sum_{y\in Y,y\neq x}}\lambda_y (xy)\\
&=\lambda_x x+\frac{1}{2} \displaystyle\sum_{y\in Y, y\neq x}\lambda_y (x+y)\\
&=\lambda_x x+\frac{1}{2} {\displaystyle\sum_{y\in Y, y\neq x}\lambda_y x}+\frac{1}{2} \displaystyle\sum_{y\in Y, y\neq x}\lambda_y y\\
&=\frac{1}{2}\lambda_x x+\frac{1}{2} {\displaystyle\sum_{y\in Y, y\neq x}\lambda_y x}+\frac{1}{2}\lambda_x x+\frac{1}{2} \displaystyle\sum_{y\in Y, y\neq x}\lambda_y y\\
&=\frac{1}{2} ({\displaystyle\sum_{y\in Y}\lambda_y) x}+\frac{1}{2} \displaystyle\sum_{y\in Y}\lambda_y y\\
&=\frac{1}{2} \displaystyle\sum_{y\in Y}\lambda_y y=\frac{1}{2}z\in I_Y.\\
\end{array}$$

If $x\not \in Y$, then $xz=0$ or
$$\begin{array}{ll}
xz
&=\displaystyle\sum_{y\in Y} \lambda_y \ xy\\
&=\displaystyle\sum_{y\in Y}\alpha_{x} \lambda_y (x+y)\\
&=\alpha_x(\displaystyle\sum_{y\in Y}\lambda_y)x+ \alpha_x \displaystyle\sum_{y\in Y}\lambda_y y\\
&= \alpha_x  \displaystyle\sum_{y\in Y}\lambda_y y =\alpha_x z\in I_Y.
\end{array}$$

Hence any subspace of  $I_Y$ is a left ideal of $A_\Gamma$.
Similarly we find every subspace of $I_Y$ to be a right ideal of $A_\Gamma$.
Clearly $A/I_Y$ is isomorphic to $A_{\Gamma/Y}$.
\end{proof}

\begin{lemma}\label{ideal}
Suppose any two vertices of $\Gamma$ are adjacent and all edges of $\Gamma$ are labeled by $\frac{1}{2-|X|}$.
Then $\langle\displaystyle\sum_{x\in X} x\rangle$ is a two-sided ideal of $A$. 
\end{lemma}

\begin{proof}
Let $x\in X$ and $\alpha=\frac{1}{2-|X|}$, then 

$$\begin{array}{ll}
x (\displaystyle\sum_{y\in X} y)
&=x+\displaystyle\sum_{y\in X,y\neq x} xy\\
&=x+\displaystyle\sum_{y\in X,y\neq x}\alpha(x+y)\\
&=x+(|X|-1)\alpha x+\alpha\displaystyle\sum_{y\in X,y\neq x} y\\
&=(|X|-1)\alpha x+(1-\alpha)x +\alpha \displaystyle\sum_{y\in X} y\\
&=\alpha \displaystyle\sum_{y\in X} y.\\
\end{array}$$
Similarly we find 
$$(\displaystyle\sum_{y\in X} y)x=\alpha \displaystyle\sum_{y\in X} y.$$
As $X$ generates $A$, we find indeed that $\langle\displaystyle\sum_{x\in X} x\rangle$ is a two-sided ideal of $A$.
\end{proof}

\begin{theorem}\label{simple}
The algebra $A$ is simple, unless  one of the following holds: 
\begin{enumerate}
\item 
There exists an ideal subgraph $\Delta$ with vertex set $Y$ of $\Gamma$ and $I_Y$ is an ideal of $A_\Gamma$.
\item $\Gamma$ is a finite clique and all edges have the same label $\frac{1}{2-|X|}$,
and $\langle \displaystyle\sum_{x\in X} x\rangle$ is a $1$-dimensional ideal.
\end{enumerate}
\end{theorem}

We prove this theorem in the following lemmas.
We fix our notation.
Let $I$ be a nontrivial two-sided ideal of $A$, i.e. $I\neq \{0\}$.

\begin{lemma}
If $X\cap I\neq \emptyset$, then $I=A$.
\end{lemma}

\begin{proof}
Assume that there is an element $x\in X\cap I$.
Then,  we find  $xy-\alpha_{x,y} x=\alpha_{x,y} y\in I$ for $y\in X$ with $(x,y)\in E$ or 
$yx-\alpha_{y,x} x=\alpha_{y,x} y\in I$ for $y\in X$ with $(y,x)\in E$. 
Using weak connectivity of $\Gamma$, we find $X\subseteq I$ and $I=A$.
\end{proof}

Now assume $X\cap I=\emptyset$. 
Let $a=\displaystyle\sum_{y\in X}\lambda_y y\in I$ be an element with $|\mathrm{supp}(a)|$ minimal among all elements in $I\setminus \{0\}$.
Notice that $|\mathrm{supp}(a)|>1$.

Then for $x\in X$ we find for all $\beta\in \mathbb{F}$ that $xa-\beta a$ as well as $ax-\beta a$ are in $I$.
We will use this for particular choices of $x$ and  $\beta$.

\begin{lemma}\label{casesfora}
One of the following holds.
\begin{enumerate}
\item 
There is a $\lambda\in \mathbb{F}^*$ with $a=\lambda\displaystyle \sum_{y\in \mathrm{supp}(a)} y$.
For any two vertices $x,y$ of $\mathrm{supp}(a)$ we find $(x,y)\in E$ with label $\frac{1}{2-|\mathrm{supp}(a)|}\neq \frac{1}{2}$ and $|\mathrm{supp}(a)|>2$.
\item

For any two vertices $x,y$ of $\mathrm{supp}(a)$ we find $(x,y)\in E$ with label $\frac{1}{2}$.
Moreover, $\displaystyle\sum_{y\in \mathrm{supp}(a)} \lambda_y=0$ in $\mathbb{F}$.
\end{enumerate}
\end{lemma}

\begin{proof}
Let $x\in \mathrm{supp}(a)$. Then we find
$$\begin{array}{ll}
xa-\beta a
&=\displaystyle\sum_{y\in \mathrm{supp}(a)} \lambda_y (xy-\beta y)\\
&=\lambda_x(1-\beta) x+\displaystyle\sum_{y\in \mathrm{supp}(a),x\sim y} \lambda_y (xy-\beta y)-\displaystyle\sum_{y\in \mathrm{supp}(a),x\not\sim y} (\beta\lambda_y y)\\
&=\lambda_x(1-\beta) x+\displaystyle\sum_{y\in \mathrm{supp}(a),x\sim y} \lambda_y (\alpha_{x,y} x +(\alpha_{x,y}-\beta) y)-\beta\displaystyle\sum_{y\in \mathrm{supp}(a),x\not\sim y} \lambda_y y.\\
\end{array}$$
This element is an element of $I$.

If $x$ has no neighbors in $\mathrm{supp}(a)$, then $xa=\lambda_x x\in I$, which is against our assumption that $X\cap I=\emptyset$.
So, $x$ does have neighbors in $\mathrm{supp}(a)$.
But then taking $\beta=0$, we find that $\mathrm{supp}(xa)$ contains the neighbors of $x$ in $\mathrm{supp}(a)$ and possibly $x$. 
As $|\mathrm{supp}(xa)|\geq |\mathrm{supp}(a)|>1$, by minimality of $|\mathrm{supp}(a)|$,
we find that all $y\in \mathrm{supp}(a)$ are either equal to $x$ or adjacent to $x$.
Varying $x$, we find that inside  $\mathrm{supp}(a)$ any two vertices are adjacent.

Moreover, when taking $\beta=\alpha_{x,z}\in \mathcal{F}$ for some $z\in \mathrm{supp}(a)$ different from $x$, we find 
that the support of
$$xa-\alpha_{x,z} x=\lambda_x(1-\alpha_{x,z}) x+\displaystyle\sum_{y\in \mathrm{supp}(a),x\sim y} \lambda_y (\alpha_{x,y} x +(\alpha_{x,y}-\alpha_{x,z}) y)$$
is contained in $\mathrm{supp}(a)\setminus \{z\}$. So, by minimality of $|\mathrm{supp}(a)|$ we find
$$\lambda_x(1-\alpha_{x,z}) x+\displaystyle\sum_{y\in \mathrm{supp}(a),y\sim x} \lambda_y (\alpha_{x,y} x +(\alpha_{x,y}-\alpha_{x,z}) y)=0,$$
implying that $\alpha_{x,y}=\alpha_{x,z}$ for all $y\in \mathrm{supp}(a)$ different from $x$ and 
$$\lambda_x(1-2\alpha_{x,z})+ \alpha_{x,z}\displaystyle\sum_{y\in \mathrm{supp}(a)} \lambda_y=0.$$

Considering
$$ax-\alpha_{z,x} x=\lambda_x(1-\alpha_{z,x}) x+\displaystyle\sum_{y\in \mathrm{supp}(a),x\sim y} \lambda_y (\alpha_{y,x} x +(\alpha_{y,x}-\alpha_{z,x}) y)$$
we also find $\alpha_{y,x}=\alpha_{z,x}$ for all $y\in \mathrm{supp}(a)$ different from $x$ and 
$$\lambda_x(1-2\alpha_{z,x})+ \alpha_{z,x}\displaystyle\sum_{y\in \mathrm{supp}(a)} \lambda_y=0.$$

In particular, if $\alpha_{x,z}\neq \frac{1}{2}$ and $\alpha_{z,x}\neq \frac{1}{2}$, then 
$$\lambda_x=\frac{\alpha_{x,z}}{2\alpha_{x,z}-1}\cdot\displaystyle\sum_{y\in \mathrm{supp}(a)} \lambda_y=\frac{\alpha_{z,x}}{2\alpha_{z,x}-1}\cdot\displaystyle\sum_{y\in \mathrm{supp}(a)} \lambda_y$$
and, as $\lambda_x\neq 0$, we find $$\frac{\alpha_{x,z}}{2\alpha_{x,z}-1}=\frac{\alpha_{z,x}}{2\alpha_{z,x}-1},$$
from which we deduce $\alpha_{x,z}=\alpha_{z,x}$.
Hence, in this case we can conclude that all edges of $\mathrm{supp}(a)$ have the same label 
$\alpha$,  
and that $\lambda_y=\lambda_x$ for all $y\in \mathrm{supp}(a)$. 
So, minimality of $|\mathrm{supp}(a)|$ implies $\lambda_x(1-2\alpha)+\alpha|\mathrm{supp}(a)|\lambda_x=0$, from which we deduce $\alpha=\frac{1}{2-|\mathrm{supp}(a)|}$. 

Notice that this also implies that  $|\mathrm{supp}(a)|>2$.

And if $\alpha_{x,z}=\frac{1}{2}$ or $\alpha_{z,x}=\frac{1}{2}$ , then $\displaystyle\sum_{y\in \mathrm{supp}(a)}\lambda_y=0$ and $\alpha_{x,z}=\alpha_{z,x}=\frac{1}{2}$.
So, in both cases we find that in $\mathrm{supp}(a)$ any two vertices are adjacent and all edges have the same label.
\end{proof}

\begin{lemma}\label{proper}
Suppose $\mathrm{supp}(a)$ is a proper subset of $X$.
Then the induced subgraph on $\mathrm{supp}(a)$ is an ideal subgraph.
\end{lemma}

\begin{proof}
Assume that $\mathrm{supp}(a)$ is a proper subset of $X$.
Then consider an element $x\in X$ which is not in $\mathrm{supp}(a)$, but adjacent to at least one vertex of $\mathrm{supp}(a)$, say $z$.
Then for $\beta\in \mathbb{F}$ we find
$$\begin{array}{ll}
xa-\beta a
&=\displaystyle\sum_{y\in \mathrm{supp}(a)} \lambda_y (xy-\beta y)\\
&=\displaystyle\sum_{y\in \mathrm{supp}(a),x\sim y}  \lambda_y(xy-\beta y)-\beta \sum_{y\in \mathrm{supp}(a),x\not\sim y} \lambda_y y\\
&=\displaystyle\sum_{y\in \mathrm{supp}(a),x\sim y} \lambda_y(\alpha_{x,y} x +(\alpha_{x,y}-\beta) y)-\beta \sum_{y\in \mathrm{supp}(a),x\not\sim y} \lambda_y y.\\

\end{array}$$

Setting  $\beta=0$, we find this product to be non-zero, and 
minimality of $|\mathrm{supp}(a)|$ implies that $x$ is adjacent to at least all but one vertices of $\mathrm{supp}(a)$.

Now assume that  $x$ is adjacent to all but one of the vertices of $\mathrm{supp}(a)$. Say $u\in \mathrm{supp}(a)$ is the unique vertex not adjacent to $x$.
Then, as $xa\neq 0$,  the support of $xa$ is of the same size  as the support of $a$.
In particular,  \cref{casesfora} also applies to $\mathrm{supp}(xa)$, so we find $\alpha_{x,y}=\alpha_{x,z}$ for all $y\in \mathrm{supp}(a)$ different from $u$. 
If both $\alpha_{x,z}$ and $\alpha_{z,u}$ are different from $\frac{1}{2}$, then, by \cref{casesfora} applied to both  $a$ and to $xa$, we can assume
that (up to a scalar) $a= \displaystyle\sum_{y\in \mathrm{supp}(a)} y$ and $xa=\lambda\displaystyle\sum_{y\in \mathrm{supp}(xa)} y$ for some nonzero $\lambda\in\mathbb{F}$.
But then $xa-\lambda a=\lambda (x-u)$ and $x(x-u)=x\in I$, which is against our assumption.
So, at least one of $\alpha_{x,z}$ or $\alpha_{z,u}$ equals $\frac{1}{2}$.

If $|\mathrm{supp}(a)|>2$, then  $\mathrm{supp}(a)$ and $\mathrm{supp}(xa)$ intersect in at least two vertices implying that all
edges inside their union have the same label $\frac{1}{2}$.
We find  $xa-\frac{1}{2} a\in I$ has support of size less than $|\mathrm{supp}(a)|$ and hence  is $0$.
But then   
$$\displaystyle\sum_{y\in \mathrm{supp}(a),y\sim x} \lambda_y =\displaystyle\sum_{y\in \mathrm{supp}(a),y\neq u} \lambda_y=0.$$

Above we found that also $\displaystyle\sum_{y\in \mathrm{supp}(a)} \lambda_y=0$
and hence $\lambda_u=0$, which is impossible.

In the case where $|\mathrm{supp}(a)|=2$ we can, without loss of generality, assume $a=z+\lambda_u u$.
But then $xa= \alpha_{x,z}(x+z)$, and $x+z\in I$.
So also $x(x+z-a)= x(x-\lambda_u u)=x\in I$, which is against our assumptions.

We conclude that $x$ is always adjacent to none or all vertices in $\mathrm{supp}(a)$.

So, assume that $x\in X$ is not in $\mathrm{supp}(a)$, but $x$ is adjacent to some fixed $z\in\mathrm{supp}(a)$
and hence to all $y\in \mathrm{supp}(a)$.
Setting  $\beta=\alpha_{x,z}$,  in the equation
$$ xa-\beta a=\displaystyle\sum_{y\in \mathrm{supp}(a),x\sim y} \lambda_y(\alpha_{x,y} x +(\alpha_{x,y}-\beta) y),$$ 
we find $xa-\beta a\in I$ to have a support of size at most 
$|\mathrm{supp}(a)|$.

If $xa-\beta a=0$, then $\alpha_{x,y}=\beta=\alpha_{x,z}$ for all $y\in \mathrm{supp}(a)$
and $$0=xa-\beta a=(\displaystyle\sum_{y\in \mathrm{supp}(a)}\lambda_y) x\in I.$$
Hence $\displaystyle\sum_{y\in \mathrm{supp}(a)}\lambda_y=0$.
If, moreover, $\alpha_{z,z'}\neq \frac{1}{2}$ for some $z,z'\in \mathrm{supp}(a)$, then, by \cref{casesfora} we have $a_{z,z'}=\frac{1}{2-|\mathrm{supp}(a)|}$ and 
there is a $\lambda\in \mathbb{F}^*$ with $a=\lambda \displaystyle\sum_{y\in \mathrm{supp}(a)} y$.
So, $\lambda |\mathrm{supp}(a)|=0$ in $\mathbb{F}$, contradicting $\alpha_{z,z'}\neq 0, \frac{1}{2}$.
Hence $\alpha_{y,y'}=\frac{1}{2}$ for all $y,y'\in \mathrm{supp}(a)$ and we find $\mathrm{supp}(a)$ to be an ideal subgraph of $\Gamma$.

If $xa-\beta a\neq 0$, then $xa-\beta a$ is also an element of $I$ with minimal support equal to $\{x\}\cup \mathrm{supp}(a)\setminus \{z\}$.
But then, by \cref{casesfora}, we have  $\alpha_{x,y}=\alpha_{x,y'}$ for all $y,y'\in \mathrm{supp}(a)$ different from $z$. 
So, for fixed $z'$ in $\mathrm{supp}(a)$ different from $z$ we find that $\alpha_{x,z'}\neq \alpha_{x,z}$ and   $xa-\alpha_{x,z'}a\neq 0$ has support contained in $\{x,z\}$.
As $xa-\alpha_{x,z'}a\neq 0$, we find that $\mathrm{supp}(xa-\alpha_{x,z'}a)=\{x,z\}$ has size $2$. But then $\mathrm{supp}(a)=\{z,z'\}$ and 
$\mathrm{supp}(xa-\beta a)=\{x,z\}$.

As $xa-\beta a$ as well as $xa-\alpha_{x,z'}a$ are in $I$ and have support of size $2$, we find, by \cref{casesfora}, that the edges $(x,z)$ and $(x,z')$ both have label $\frac{1}{2}$, contradicting that they have different label.

If there is no $x\in X$  which is not in $\mathrm{supp}(a)$ but  adjacent to an element from $\mathrm{supp}(x)$, then, by the weakly connectedness
of $\Gamma$ there is an $x\in X$ but not in $\mathrm{supp}(a)$ for which there is a $z\in \mathrm{supp}(a)$ adjacent to $x$.
Considering products of the form $ax-\beta a$ we obtain the same conclusions as above.
\end{proof}

\cref{simple} now  follows.
Indeed, if $\mathrm{supp}(a)=X$, then \cref{casesfora} clearly implies that either we are in case (ii) of the conclusion of the theorem and $\langle \sum_{x\in X}x\rangle$ is an ideal as follows from \cref{ideal},
or we are in case (i) and $X$ is an ideal subgraph and $I_X$ is an ideal of $A$, as in \cref{idealgraph}.
If $\mathrm{supp}(a)\neq X$, then \cref{proper} implies that $\mathrm{supp}(a)$ is an ideal subgraph and, 
by \cref{idealgraph} we find $I_{\mathrm{supp}(a)}$ to be a proper ideal of $A$.

For later use, we mention the following obvious corollary to \cref{simple}:

\begin{corollary}\label{simplecor}
If $\Gamma$ is weakly connected and has  at least three vertices, but no complete subgraph on three vertices,
then $A$ is simple.
\end{corollary}

\begin{remark}\label{allen}
Suppose $\Gamma=(X,E)$ is a finite complete graph in which each edge is labeled $\alpha=\frac{1}{2-|X|}$, where $|X|\neq 2$ in $\mathbb{F}$,
and $A=A_\Gamma$.
Then the algebra $A/I$ where $I=\langle\displaystyle \sum_{x\in X} x\rangle$
is generated by elements $\overline{x}=(n-1)x+I$ with $x\in X$ and $n=\mathrm{dim}(A/I)=|X|-1$.

If $a\neq b\in X$, then 
$$\begin{array}{ll}
\overline{a}^2&= (n-1)\overline{a},\ \mathrm{and}\\
\overline{a}\overline{b}&=-(\overline{a}+\overline{b}).
\end{array}$$

The algebras $A/I$ have been studied by Harada \cite{harada} and Allen \cite{allen1,allen2} in connection with $2$-transitive groups. 
If the characteristic $p$ of $\mathbb{F}$ does not divide $|X|$, then Allen proves the automorphism group of $A_\Gamma/I$ to be isomorphic 
to the symmetric group $\mathrm{Sym}_{n+1}$, which is also isomorphic to the automorphism group of the graph $\Gamma$.

If  the characteristic $p$ of $\mathbb{F}$ divides $|X|$, then $\frac{1}{2-|X|}=\frac{1}{2}$ and the algebra $A$ not only has a $1$-dimensional ideal,
but also $X$ is an ideal subgraph and $I_X$ is an ideal of codimension $1$.
Allen \cite{allen1,allen2} shows the algebra $A/I$ to be a special Jordan algebra (and   a Bernstein algebra)
with automorphism group isomorphic to a split extension $\mathbb{F}^n:\mathrm{GL}_n(\mathbb{F})$.

This does show that the automorphism group of the algebra $A=A_\Gamma$ can really be bigger than that of the graph $\Gamma$.
\end{remark}


\section{Idempotents and fusion rules for algebras from graphs}
\label{section:axial}

As before, let $\mathcal{F}$ be a subset of the field $\mathbb{F}$ containing  $1$ and at least one other element 
and $\Gamma=(X,E)$  a weakly connected directed graph, with vertex set $X$ and edge set $E$ and all edges  labeled
with non-zero elements from the set $\mathcal{F}$ and consider the algebra  $A:=A_{\Gamma}$.

\begin{proposition}\label{axialprop}
Let $x$ be an element in $X$. Suppose the labels on  the edges with tail  $x$ are the elements of the subset $\mathcal{F}_{L_a}$
of $\mathbb{F}$  and those with head $x$ are the elements of the subset $\mathcal{F}_{R_x}$. 
If $1\not \in \mathcal{F}_{L_a}\cup \mathcal{F}_{R_a}$, then the algebra $A$ is the direct sum 
$$A=A_1(L_x)\oplus A_0(L_x)\oplus \bigoplus_{\alpha\in \mathcal{F}_{L_x}} A_{\alpha}(L_x)$$ of the following   $L_x$-eigenspaces (provided they are not $\{0\}$)
$$
\begin{array}{ll}
A_{1}(L_x)&=\langle x\rangle,\\
A_{\alpha}(L_x)&=\langle \alpha x+ (\alpha-1)y\mid x\neq y\ \mathit{adjacent}\ \mathit{and}\ \alpha_{x,y}=\alpha\rangle, \mathit{with}\ \alpha\in \mathcal{F}_{L_x}\  \mathit{and}\\
A_{0}(L_x)&=\langle y\mid x\neq y\ \mathit{nonadjacent}\rangle,\\

\end{array}$$
and $$A=A_1(R_x)\oplus A_0(R_x)\oplus \bigoplus_{\alpha\in \mathcal{F}_{R_x}} A_{\alpha}(R_x)$$ of the following   $R_x$-eigenspaces (provided they are not $\{0\}$)
$$
\begin{array}{ll}
A_{1}(R_x)&=\langle x\rangle,\\
A_{\alpha}(R_x)&=\langle \alpha x+ (\alpha-1)y\mid y\neq x\ \mathit{adjacent}\ \mathit{and}\ \alpha_{y,x}=\alpha\rangle, \ \mathit{with}\ \alpha\in \mathcal{F}_{R_x}\ \mathit{and}\\
A_{0}(R_x)&=\langle y\mid y\neq x\ \mathit{nonadjacent}\rangle.\\

\end{array}$$

In particular, if all labels of the graph $\Gamma$ are different from $1$ then  $(A,X)$ is a primitive axial algebra. 
\end{proposition}

\begin{proof}
Let $x\in X$. We consider the map $L_x$.
The algebra  $A$ clearly decomposes as $A_1(L_x)\oplus A_0(L_x)\oplus \bigoplus_{\alpha\in \mathcal{F}} A_{\alpha}(x)$.
For each $\alpha\in \mathcal{F}_{L_x}$ we find that $A_{\alpha}(L_x)$ is contained in the  $\alpha$-eigenspace for $L_x$  as
for all adjacent $x$ and $y$ with $\alpha_{x,y}=\alpha$ and all nonadjacent $x$ and $z$ we have
$$\begin{array}{l}
L_x(x)=x^2=x,\\
L_x(\alpha x+(\alpha-1)y)=x(\alpha x+(\alpha-1)y)=\alpha x+(\alpha-1)\alpha (x+y)=\alpha(\alpha x+(\alpha-1)y),\\
L_x(z)=xz=0.\\
\end{array}$$

The proof for $R_x$ is similar.
\end{proof}

If $x,y\in X$ are on the edge $(x,y)$ with label $1$, then 
$L_x$ as well as $R_y$ are  not semi-simple.
The subspace $\langle x,y\rangle$ is clearly invariant under $L_x$ and $R_y$, but the $1$-eigenspace of $L_x$ and $R_y$ on this subspace
is just $\langle x\rangle$ or $\langle y\rangle$, respectively.

So, $A$ is an axial algebra with respect to its generating set $X$ of idempotents, provided none of the edges has label $1$.

In the sequel of this section we assume that no edge is labeled by $1$ and  determine the fusion laws for  the decompositions given in
\cref{axialprop}.

We set $\alpha_{x,y}=0$ in case $x$ and $y$ are two nonadjacent vertices in $ X$.
So, for distinct $x,y\in X$ we then have $xy=\alpha_{x,y}(x+y)$.

\

\begin{lemma}\label{fusionlemma}
Suppose $\Gamma$ contains at least $3$ vertices.
Let $\alpha,\beta\in \mathcal{F}$ be different from $1$.
Then, for $K\in \{L,R\}$ and $x\in X$ we have the following:
\begin{enumerate}
\item If $\alpha=0$, then $A_\alpha(K_x)A_\alpha(K_x)\subseteq A_\alpha(K_x)$.
\item If $\alpha\neq 0$, then $A_\alpha(K_x)A_\alpha(K_x)\subseteq A_1(K_x)+A_\alpha(K_x)$. 
\item If $\alpha\neq \beta$ 
then $A_\alpha(K_x)A_\beta(K_x)\subseteq A_1(K_x)+A_\alpha(K_x)+A_\beta(K_x)$.
\end{enumerate}
\end{lemma}

\begin{proof}
Let $x\in X$. We consider $K=L$.
Then for $y,z\in A_0(L_x)\cap X$ we find  $xy$ to be a linear combination of $x$ and $y$ and hence in $A_0(L_x)$.
As $A_0(L_x)$ is generated by its elements from $X$, this proves (i) for $K=L$.

Now assume $\alpha\neq 0, 1$ in $\mathcal{F}$ and consider $A_\alpha(L_x)$.
Then this space is spanned by elements of the form $\alpha x+(\alpha-1)y$, where $(x,y)$ is an edge with label $\alpha$.
But then the product of any two such elements is a linear combination
of $x$ and such $y$. In particular, $A_\alpha(K_x)A_\alpha(K_x)\subseteq A_1(K_x)+A_\alpha(K_x)$, proving (ii).

Moreover, the product of an element $\alpha x+(\alpha-1)y$ with an element $z\in X$ not adjacent to $x$ is
a linear combination of $y$ and $z$ and thus inside $A_1(K_x)+A_\alpha(K_x)+A_0(K_x)$.
Thus $A_\alpha(K_x)A_0(K_x)\subseteq A_1(K_x)+A_\alpha(K_x)+A_0(K_x)$.
Similarly we find $A_0(K_x)A_\alpha(K_x)\subseteq A_1(K_x)+A_\alpha(K_x)+A_0(K_x)$. Hence we have proven (iii) in case $K=L$ and $\beta=0$.

Now assume that  $\alpha=\alpha_{x,y}\neq \beta_{x,z}=\beta$ are both different from $0,1$.
Then the product 
$$(\alpha x+(\alpha-1)y)(\beta x+(\beta-1)z)$$
is a linear combination of $x,y$ and $z$ and thus contained in  $A_1(L_x)+A_\alpha(L_x)+A_\beta(L_x)$, proving (iii) in case $K=L$ and $\beta\neq 0$.

The case that $K=R$ is similar.
\end{proof}

The above lemmas lead to the following theorem:

\begin{theorem}\label{fusionthm}
Let 
$\Gamma$ be a directed graph with edges labeled by elements  different from $0$ and $1$ from a subset $\mathcal{F}\subseteq \mathbb{F}$ containing $0$ in case $\Gamma$ is not complete,  $1$ and at least one other element.  
Then the  algebra $A_\Gamma$ is an axial algebra with axes in $X$ satisfying for each $x\in X$ the  fusion rules with  respect to the decompositions
$$\bigoplus_{\alpha\in \mathcal{F}} A_\alpha(L_x)$$
and $$\bigoplus_{\alpha\in \mathcal{F}} A_\alpha(R_x)$$
given in \cref{Fusion rules for graphs}.
\end{theorem}

\begin{table}[h!]

$$\begin{array}{|l||l|l|l|l|l|}
\hline
\star & 1 &  \alpha&\dots &\beta &(0)\\
\hline\hline
1& 1&\alpha& & \beta&0\\
\hline
\alpha& \alpha& 1,\alpha & &1,\alpha,\beta&1,\alpha,0\\
\hline
\vdots  &&&&&\phantom{a}\\
\hline
\beta &\beta& 1,\alpha,\beta&  & 1,\beta& 1,\beta,0 \\
\hline
(0)     &   0 & 1,\alpha,0&& 1,\beta,0 &0\\
\hline
\end{array}
\hspace{1cm}
\begin{array}{|l||l|l|l|}
\hline
\star & 1 &  \alpha& (0)\\
\hline\hline
1& 1&\alpha&0\\
\hline
\alpha& \alpha& 1,\alpha & 1,\alpha, 0\\
\hline
(0)     &   0 & 1,\alpha,0&0\\
\hline
\end{array}$$


\caption{Fusion law of graph type $\mathcal{G}(\mathcal{F})$ where $\mathcal{F}=\{1,\alpha,\dots,\beta,(0)\}$ and  $\mathcal{F}=\{1,\alpha,(0)\}$, where $\alpha\neq \beta$ are different from $0$ and $1$.}\label{Fusion rules for graphs}
\end{table}

The fusion laws displayed in \cref{Fusion rules for graphs} on a set $\mathcal{F}$ will be denoted by $\mathcal{G}(\mathcal{F})$ and called  fusion laws of \emph{graph type}.

\cref{algebrathm} follows from \cref{simple} and \cref{fusionthm}.

\section{Idempotents and elements of small rank}
\label{section:idempotent}
Suppose $\Gamma=(X,E)$ is a symmetric weakly connected graph  with edges labeled by nonzero elements 
from a set $\mathcal{F}\subseteq \mathbb{F}$, which we assume to contain  $1$ and also $0$ in case $\Gamma$ is not complete. (We allow edges to be labeled by $1$.)
Let $A=A_{\Gamma}$ be the corresponding  algebra.

Each symmetric directed graph can be obtained from a simple undirected graph 
by replacing edges $\{x,y\}$  by pairs of edges $(x,y)$ and $(y,x)$. 
Of course we can also reverse this process and replace in a symmetric directed graph each directed edge $(x,y)$ by an undirected edge $\{x,y\}$.
This yields a one-to-one relation between symmetric directed graphs and undirected graphs.   
The undirected graph obtained from $\Gamma$ will be called the underlying graph of $\Gamma$ and denoted by $\underline{\Gamma}$.

Now consider  the algebra $A_\Gamma$. Symmetry of $\Gamma$ is equivalent with for  $x\neq y\in X$ that  $xy\neq 0$ if and
and only if $yx\neq 0$.
For each $x\in X$ we find the left and right adjoint maps $L_x$ and $R_x$ to be a linear  map from $A$ to $A$ with eigenvalues in $\mathcal{F}$.
The image of  $L_x$ is the subspace of $A$ spanned by $x$ and all its neighbors.
The image of $R_x$ is spanned by $x$ and all $y$ for which $x$ is a neighbor. 
So the  rank of $L_x$ equals  $1+\delta^+_x$, where $\delta^+_x$ denotes the outdegree of the vertex $\in X$ and the rank of $R_x$
equals $1+\delta^-_x$, where $\delta^-_x$ is the indegree of $x$. Notice that these degrees may be infinite.
Since we assume $\Gamma$ to be symmetric we obtain that the indegree and outdegree of a vertex $x\in X$ are equal and the left and right adjoint maps $L_x$ and $R_x$
have the same rank.

Now assume $a=\displaystyle\sum_{y\in X} \lambda_y y$ to be an arbitrary element in $A$  with the adjoint maps $L_a$ of rank $r^+_a$
and $R_a$ of rank $r^-_a$ .
We will  relate some properties of the subgraph of $\Gamma$ induced on $\mathrm{supp}(a)$ to the ranks $r^+_a$ of $L_a$
and $r^-_a$ of $R_a$.
We identify $\mathrm{supp}(a)$ with the subgraph induced on it by the underlying graph $\underline{\Gamma}$ of $\Gamma$.
Notice that in the graph induced on $\mathrm{supp}(a)$ two vertices $x\neq y$ are adjacent if and only if their product $xy$ is nonzero.

\begin{lemma}\label{components}
If $a$ is an idempotent with $L_a$ or $R_a$ of rank $r$, then $\mathrm{supp}(a)$ has at most $r$ connected components.

If $a$ is an idempotent and the $1$-eigenspace of $L_a$ (or $R_a$) has dimension $m$, then $\mathrm{supp}(a)$ has at most $m$ components.
In particular, if $a$ is a primitive idempotent, then $\mathrm{supp}(a)$ is strongly connected.
\end{lemma}

\begin{proof}
Suppose $C_1,\dots, C_n$ are the connected components of  $\mathrm{supp}(a)$
and write $a=a_1+\dots +a_n$ with $\mathrm{supp}(a_i)=C_i$.
Then $L_a(a_i)=a_i$. 
So the rank of $L_a$ is at least $n$.
This proves the first part.

Moreover, we find each $a_i$ to be a eigenvector for $L_a$ with eigenvalue $1$, and $n\leq m$, proving the second statement.
So, if $a$ is a primitive idempotent, then the eigenspace  for eigenvalue $1$ is $1$-dimensional.
Hence $\mathrm{supp}(a)$ needs to be connected.

The proof for $R_a$ is similar.
\end{proof}

\begin{lemma}\label{treeindep}\label{treebound}
Suppose $Y\subseteq \mathrm{supp}(a)$ such that the induced subgraph on $Y$ by the underlying graph of $\Gamma$ is a tree.
Let $y_0$ be a leaf. Suppose $y_1, \dots, y_n$ are the vertices of $Y\setminus L$, where $L$ is the set of leaves of $Y$.
Then $ay_0,\dots, ay_n$ (as well as $y_0a,\dots, y_na$) are linearly independent.

In particular,  $L_a(Y)$ (as well as $R_a(Y)$) has dimension  at least $|Y|+1-\ell$,
where $\ell=|L|$.
\end{lemma}

\begin{proof}
Let $Y\subseteq \mathrm{supp}(a)$ and assume that the induced subgraph on $Y$ is a tree.
Fix $y_0$ and order the elements $y_1, \dots, y_n$ of $Y\setminus L$ where $L$ is the set of leaves of $Y$
in such a way that for $1\leq j\leq k\leq n$ we have that the distance in the tree of $y_i$ to $y_0$ is at most the distance of $y_k$ to $y_0$
(First number all elements at distance $1$, then at distance $2$, etc.)
Let $y_{n+1},\dots, y_{n+\ell-1}$ be the  leaves of the tree induced on $Y$ different from $y_0$.
Then for each $1\leq j\leq n$, the element $y_j$ is not a leaf and we find
an element $y_k$ with $k>j$
adjacent to it.
As $Y$ is a tree, this element $y_k$ is not adjacent to $y_i$ for all $i< j$.
But then  for $i< j$ the support of $ay_i$ does not contain $y_k$, while the support of
$ay_j$ does contain $y_k$.
In particular, we find that the elements $ay_0,\dots, ay_n$ are linearly independent.

The proof that $y_0a,\dots, y_na$ are linearly independent, is similar.
\end{proof}

The \emph{girth} $g$ of the underlying graph of  $\Gamma$ is the length of a minimal cycle in this graph.
If the underlying graph of $\Gamma$ does not contain a cycle, then $g$ is set to be $\infty$.



The above lemma implies the following.

\begin{corollary}\label{tree}
Let $r_a$ be equal to $r^+_a$ or $r^-_a$. 
If $r_a\leq g-3$, then  $\mathrm{supp}(a)$ is a forest. 
If this forest consists of the trees $T_1,\dots, T_t$, with sets of leaves $L_1,\dots ,L_t$ respectively,
then $\displaystyle\sum_{i=1}^t (|T_i|-|L_i|+1)\leq r_a\leq g-3$.
\end{corollary}

\begin{proof}
Suppose $r_a^+\leq g-3$.
Let $C$ be a minimal cycle in  $\mathrm{supp}(a)$. Then removing from $C$ a vertex in  yields a path on at least $g-1$ vertices, 
which by \cref{treeindep} implies that $L_a$ has rank at least $g-2$, contradicting our assumptions.
So $\mathrm{supp}(a)$ is a forest 
consisting of trees $T_1, \dots, T_t$, say.
Then the subspace spanned by the vertices of $\mathrm{supp}(a)$ is a direct sum of the various subspaces $V_i$ spanned by the vertices of $T_i$. 
By \cref{treebound} we find  for each $i$ the subspace $L_a(V_i)$ to have  dimension at least $|T_i|-|L_i|+1$. As $L_a(V_i)$ is contained in $V_i$, 
it follows that $r^+_a$  is at least $\displaystyle\sum_{i=1}^t (|T_i|-|L_i|+1)$.

Similar arguments apply in case $r^-_a\leq g-3$.
\end{proof}

\begin{lemma}\label{specialtree}
Suppose the underlying graph of $\Gamma$ has finite girth $g>4$.
If $a$ is an idempotent with $r^+_a=g-2$ or $r^-_a=g-2$, then $\underline{\Gamma}$ is a cycle, 
or $\mathrm{supp}(a)$ is a forest.
\end{lemma}

\begin{proof}
Suppose $\mathrm{supp}(a)$ is not a forest and hence contains a cycle.
Assume $r^+_a=g-2$.
Then $\mathrm{supp}(a)$ contains a path $a_0\sim\cdots \sim a_{g-2}$ along the cycle.
By \cref{treeindep}  we find  the elements 
$L_a(a_i)$, with $0\leq i< g-2$ to be  linearly independent.
As $r^+_a$ equals $g-2$ we find that these elements span the range of $L_a$ and, as $a$ is an idempotent,  $a$ is in this span.

So, $a$ is a linear combination of $a_0,\dots, a_{g-3}$ and their neighbors, and $\mathrm{supp}(a)$
is contained in the set of these vertices together with their neighbors.
As $\mathrm{supp}(a)$ contains a cycle on $a_0,\dots, a_{g-3}, a_{g-2}$, and can only contain neighbors of $a_0,\dots, a_{g-3}$,
it contains a cycle on $a_0,\dots, a_{g-3}, a_{g-2}, a_{g-1}$, where $a_{g-1}$ is a common neighbor of $a_0$ and $a_{g-2}$.

As $g>4$, we find that $a_{g-1}$ is the unique common neighbor of $a_0$ and $a_{g-2}$.

None of the neighbors of $a_{g-1}$ different from $a_{g-2}$ and $a_0$ is a neighbor of any of the other vertices
of the cycle and hence not in the support of $a$.
Moreover, if $b$ is a neighbor of $a_{g-1}$ different from $a_{g-2}$ and $a_0$, then 
$ab$ is a nonzero multiple of $a_{g-1}b$, which is not in the span of the elements $aa_i$, with $0\leq i<g-2$.
But these elements  span the image of $L_a$, as we saw above.
This leads to a contradiction, and we conclude that $a_{g-1}$ has only the neighbors $a_0$ and $a_{g-2}$.

Repeating this argument starting with a different path along the cycle, we find that
the support of $a$ only contains the vertices of the cycle and none of them has a neighbor outside the cycle.
We conclude that $\mathrm{supp}(a)=X$ and, as the underlying graph of $\Gamma$ is connected, it is a cycle.

If $r^-_a=g-2$, we can apply a similar argument using $R_a$ instead of $L_a$. 
\end{proof}

\begin{lemma}\label{leaves}
Suppose $a$ is an idempotent
and $\mathrm{supp}(a)$ a tree with $n$ vertices and diameter $d$.
Let $L$ be the set of leaves of $\mathrm{supp}(a)$ and set $\ell=|L|$.
Then we have the following:
\begin{enumerate}
\item 
If  $d\leq g-3$, then $r^\pm_a\geq n-\ell+1+\displaystyle\sum_{l\in L}(\delta_l-1)$.
\item 
If $d=g-2\geq 4$  then for any two leaves $a_0$ and $a_d$ at distance $d$ in $\mathrm{supp}(a)$ we have  $r^\pm_a\geq d+\delta_{a_0}+\delta_{a_d}-4$.  
\end{enumerate}
\end{lemma}

\begin{proof}
Suppose $\mathrm{supp}(a)$ is a   tree  with $n$ vertices, $\ell$ leaves and diameter $d$.

If $d\leq g-3$, a vertex $b\in X\setminus\mathrm{supp}(x)$ is adjacent in $\underline{\Gamma}$ to at most one vertex of $\mathrm{supp}(a)$.
For  each leaf $x\in \mathrm{supp}(a)$ and neighbor $y$ outside $\mathrm{supp}(a)$, 
we find $ay=\lambda_x xy=\lambda_x\alpha_{x,y} (x+y)$ to be in the range of $L_a$
and $ya=\lambda_x yx=\lambda_x\alpha_{y,x} (x+y)$ to be in the range of $R_a$.

As each leaf $l$ of $\mathrm{supp}(a)$ has at least $\delta_l-1$ neighbors outside $\mathrm{supp}(a)$, we find 
$L_a(\langle X\setminus\mathrm{supp}(a)\rangle)$ and $R_a(\langle X\setminus\mathrm{supp}(a)\rangle)$ to have a subspace of dimension at least $\displaystyle\sum_{l \in L}(\delta_l-1)$ intersecting $\langle \mathrm{supp}(a)\rangle$ trivially.
As we saw in \cref{treebound} the spaces  $L_a(\langle \mathrm{supp}(a)\rangle)$ and  $R_a(\langle \mathrm{supp}(a)\rangle)$ have dimension at least $n-\ell+1$
and are contained in $\langle \mathrm{supp}(a)\rangle$.
Hence $r^\pm_a\geq n-\ell+1+\displaystyle\sum_{l\in L}(\delta_{l}-1)$.

If $d=g-2\geq 4$, then consider two leaves $a_0$ and $a_d$ of $\mathrm{supp}(a)$ at distance $d$
and take a path $a_0\sim\dots \sim a_d$ from $a_0$ to $a_d$ inside $\mathrm{supp}(a)$.
As $g\geq 6$, the elements $a_0$ and $a_d$ have at most one common neighbor.
So, we can find $\delta_{a_0}-2$ neighbors $x$ of $a_0$ 
which are not adjacent to  any of the vertices $a_i$ with $i>0$. 
Similarly we find  $\delta_{a_d}-2$ vertices $y\in X$ which are only adjacent to $a_d$, but not any of the vertices $a_i$ with $i<d$.
Then $aa_0,\dots, aa_{d-1}$ (or $a_0a,\dots, a_{d-1}a$)
together with the various $ax=a_0x$, and $ay=a_dy$ (or $xa=xa_0$, and $ya=ya_d$)  are linearly independent, proving $r^\pm_a$ to be at least $d+\delta_{a_0}+\delta_{a_d}-4$.  
\end{proof}

\section{Automorphism groups of the algebras}
\label{section:auto}

As before, let $\mathcal{F}$ be a subset of the field $\mathbb{F}$ containing $0,1$, and $\Gamma=(X,E)$ a weakly connected directed graph
with all edges labeled by nonzero elements from $\mathcal{F}$. 
Then the algebra $A=A_\Gamma$ is considered to be an algebra over the field $\mathbb{F}$.

In this section will provide some results in which we present some conditions under which the automorphism group of the graph $\Gamma$ and 
the algebra $A_\Gamma$ are isomorphic. We assume the graph $\Gamma$ to be symmetric.

Denote by $k_{\mathrm{max}}$ the maximum of $\{\delta^\pm_x\mid x\in X\}$ and $k_{\mathrm{min}}$ its minimum, or in case this maximum or minimum does not exist, set $k_{\mathrm{max}}$ or $k_\mathrm{min}$ to be equal to $\infty$.

We first prove \cref{graphthm}:

\begin{theorem}\label{autgroupthm}
Let $\Gamma$ be a weakly connected symmetric directed graph labeled with values different from $0$ and $1$, such that its underlying graph has  finite girth $g$,
and $2<k_\mathrm{min}\leq g-3$, 
and $k_\mathrm{max}\leq\mathrm{min}(2(k_\mathrm{min}-1),g-3)$.

Then $\mathrm{Aut}(A)$ is isomorphic to
the automorphism group $\mathrm{Aut}(\Gamma)$ of the labeled graph $\Gamma$.
\end{theorem}

\begin{proof}
Clearly $\mathrm{Aut}(\Gamma)$, the automorphism group of $\Gamma$ as labeled graph, embeds injectively into the  group of automorphisms of $A$.
So, it remains to prove that $\mathrm{Aut}(A)$ embeds injectively into $\mathrm{Aut}(\Gamma)$

Suppose the underlying graph of $\Gamma$ is a graph with girth $g$ and $2<k_\mathrm{min}\leq g-3$, 
and $k_\mathrm{max}\leq\mathrm{min}(2(k_\mathrm{min}-1),g-3)$.

Take $h\in \mathrm{Aut}(A)$  and $x_0\in X$ and set $a=x_0^h$.
The element $x_0$ is a primitive idempotent and both $L_{x_0}$ and $R_{x_0}$ are of  rank at most $1+k_\mathrm{max}$, as follows from \cref{axialprop}.
But then $a$ is also a primitive idempotent and the  ranks $r^+_a$ of $L_a$ and $r^-_a$ of $R_a$   are at most $1+k_\mathrm{max}$ which is at most $g-2$.
Let  $r_a$ be one of  $r^+_a$ or $r^-_a$.
If $r_a\leq g-3$, then by \cref{components,tree}, the support of $a$ is a tree.
If $r_a=g-2$, then \cref{specialtree} applies, and as each vertex has degree at least $3$, we can again conclude
the support of $a$ to be a tree.
If this tree has $n$ vertices and  $\ell\geq 2$ leaves and diameter $\leq g-3$, 
then  \cref{leaves} implies that $r_a\geq n-\ell +1+\ell(k_\mathrm{min}-1)$. 
Since $r_a\leq 1+k_{\mathrm{max}}\leq 2k_{\mathrm{max}}-1$, we find $n=\ell=1$.

If this tree has  has diameter $g-2$, then $r_a$ is at least $g$ by \cref{leaves} and we get a contradiction.
We conclude that $\mathrm{supp}(a)$ is of size $1$.
So $a=\lambda x$ for some $x\in X$ and $\lambda\in \mathbb{F}$.
But then $a^2=a$ implies that $a=x\in X$.

So, we can conclude that $h(X)\subseteq X$. Moreover, as $X$ is a basis of $A$ and $h$ an automorphism of $A$, we find $h(X)=X$.
Hence $\mathrm{Aut}(A)$ leaves the set $X$ of primitive axes invariant and induces a permutation action on  $X$.  
Moreover, as $x,y\in X$ are adjacent and the edge has label $\alpha_{x,y}$ if and only if their product $xy=\alpha_{x,y}(x+y)\neq 0$, 
and $X$ is a basis for $A$, we find that $\mathrm{Aut}(A)$ embeds injectively into the automorphism group of the labeled graph $\Gamma$.
This proves the isomorphism of the two groups.
\end{proof}

\begin{example}
As already noticed in the introduction, various graphs with constant valency $3$ and girth at least $6$ with interesting automorphism group exist.
We can apply the above theorem to obtain for these graphs algebras with interesting automorphism group.

We give some examples:
Taking for $\Gamma$ the symmetric directed graph with underlying graph the    Heawood or Coxeter graph, both with automorphism group $\mathrm{PSL}_3(2):2$,
Tutte's 8-cage with automorphism group $\mathrm{Sym}_6.2$, the Foster graph with  automorphism group $3\cdot\mathrm{Sym}_6.2$ or the Biggs-Smith graph with automorphism group $\mathrm{PSL}_2(17)$ and labeling its edges with a fixed  $\alpha\in \mathbb{F}$ different from $0,1$ provides such commutative axial algebras $A_\Gamma$. See for example \cite{drg} for a description of these graphs.
\end{example}

\begin{example}
Let $G$ be a group and $S$ a set of generators of size at least $3$, such that for each $s\in S$ we also have $s^{-1}\in S$.
Then let $\Gamma$ be the Cayley graph of $G$ with respect to this set $S$ of generators.
Now fix a map $\alpha:S\rightarrow \mathbb{F}\setminus \{0,1\}$ and label the edges $(g,gs)$ with $\alpha(s)$.
Then we obtain a symmetric directed graph $\Gamma$ with constant degree $|S|$.
The girth of the underlying graph is the length of the shortest reduced word in $S$ which is the identity in $G$.

In \cite{bray,cayley} one can find various examples of finite groups $G$ with small generating  $S$ (e.g. of size  $3$, $4$  or $5$) and for which  the 
corresponding Cayley graph has girth $g\geq |S|+2$.
In particular if the map $\alpha$ is injective, we find $G$ to be the automorphism group of $\Gamma$ and by the above result
also the automorphism group of $A_\Gamma$.
\end{example}

A \emph{partial linear space} $\Delta=(P,L)$ consists of a set $P$ of \emph{points}, and set $L$ of \emph{lines} which are subsets of $P$ of size at least $2$, such that each pair of distinct points is on a unique line.
So, an ordinary graph is a partial linear space in which all lines have size $2$.

The \emph{directed incidence graph} of a partial linear space $\Delta=(P,L)$ is the directed graph with vertex set $X=P\cup L$ and two vertices adjacent if one of them is a point and the other a line containing the point.
Clearly a directed incidence graph is symmetric.

\begin{example}
The directed incidence graph $\Gamma$ of a generalized quadrangle  of order $(s,s)$ has valency $s+1$ and girth $8$. (See for example \cite{drg}.) 
So, for $2\leq s\leq 4$ we can apply \cref{autgroupthm} and after labeling the edges $(x,y)$ of this incidence graph where 
$x$ is a point and $y$ a line with some fixed $\alpha$ different from $0,1$ and the edges $(a,b)$ where $a$ is a line and $b$ a point with a fixed $\beta$
different from $0,1$
we obtain an algebra $A_\Gamma$ whose automorphism group is isomorphic with that of the labeled graph $\Gamma$.
This algebra is commutative if $\beta=\alpha$, but non-commutative otherwise. 

The (directed) \emph{incidence graph} $\Gamma$ of a generalized hexagon  of order $(s,s)$ has valency $s+1$ and girth $12$.
So, here we can apply the above to obtain commutative as well as non-commutative algebras, which by  \cref{autgroupthm} for $2\leq s\leq 8$ have  automorphism group  isomorphic to that of the labeled graph $\Gamma$.
\end{example}

Our approach to prove  \cref{autgroupthm} is based on the fact that we can identify idempotents $a$ in the algebra $A_\Gamma$  for which the
rank of $L_a$ or $R_a$ is small.
We now exploit that method to identify a significant set of idempotents as elements of $X$ in case $X$ is the vertex set of an incidence graph and prove \cref{incgraphthm}.
But first we derive some well known properties of incidence graphs. 

\begin{lemma}\label{iso_of_groups}
Let $\Gamma$ be the directed incidence graph of a graph $\Delta$ containing a vertex of degree at least $3$, or a connected partial linear space with three points per line and a point which is on at least four lines.

Then the following holds:
\begin{enumerate}
\item The underlying graph of $\Gamma$ has girth at least $6$;
\item $\mathrm{Aut}(\Gamma)$ is isomorphic to $\mathrm{Aut}(\Delta)$.
\end{enumerate}
\end{lemma}

\begin{proof}
If the underlying graph of  $\Gamma$ contains a cycle of length $\leq 5$, then as this graph is bipartite, it is of length $4$ and
we find a pair of points being together on two lines. This contradicts the definition of partial linear space.
Hence (i) follows.

Suppose $\ell$ is a line of $\Delta$. Then in $\Gamma$ this line $\ell$ has degree $|\ell|$ (so, $2$ or $3$).
A point  of $\Delta$ which is on at least $|\ell|+1$ edges or lines in $\Delta$ has degree at least $|\ell|+1$ in $\Gamma$.

Take a vertex $p$ of $\Gamma$ with degree at least $|\ell|+1$. Then the points (or vertices) of $\Delta$ are the vertices of 
$\Gamma$ which are at even distance from $p$. The lines of $\Delta$ are the vertices $\Gamma$ at odd distance from $p$.
Moreover a point is on a line if and only if they are adjacent in $\Gamma$.
So, we can reconstruct $\Delta$ from $\Gamma$ and find their automorphism groups to be isomorphic. This proves (ii).
\end{proof}

\begin{theorem}\label{autgroupthm2}
Suppose $\Delta=(Y,F)$ is a connected graph with  $k_{\mathrm{min}}\geq 3$ or a connected 
partial linear space with three points per line and at least four lines per point.
Let $\Gamma$ be the directed incidence graph of $\Delta$.
Then for any labeling of $\Gamma$ with elements different from $0$ and $1$, 
we find  that $\mathrm{Aut}(A)$ is isomorphic to $\mathrm{Aut}(\Gamma)$ the automorphism group of the labeled graph $\Gamma$.
\end{theorem}

\begin{proof}
As before, $\mathrm{Aut}(\Gamma)$ embeds injectively into $\mathrm{Aut}(A)$. So, we will
show that $\mathrm{Aut}(A)$ embeds injectively into $\mathrm{Aut}(\Gamma)$.
Notice that the girth $g$ of $\Gamma$ is at least $6$.

Let $X=Y\cup F$ be the vertex set of $\Gamma$ and fix an element $f\in F$ and $h\in \mathrm{Aut}(A)$.
Set $a=f^h$.

First assume that $\Delta$ is a graph.
As $f$ has two  neighbors in $\Gamma$, the rank of $L_f$ and $R_f$ equals   $3$.
Since none of the labels equals $1$, the element $f$ is a primitive idempotent. 
So, also $a$ and $L_a$ and $R_a$ have  these properties.
By \cref{components} we find $\mathrm{supp}(a)$ to be connected.
As $r_a=3\leq g-2$, we find by \cref{tree,specialtree} that $ \mathrm{supp}(a)$ is a tree.
Moreover, by \cref{treeindep}, this tree is of diameter $d\leq 3$.

If $d\geq 2$, then \cref{leaves} implies  that $r_a>3$,   leading to a contradiction.
If $d=1$, then $\mathrm{supp}(a)$ is an edge with two vertices, and  there are  at least $3$ neighbors of these vertices outside $\mathrm{supp}(a)$
that are adjacent to only one vertex in $\mathrm{supp}(a)$. But then 
we find that the image of $L_a$ has dimension at least $4$, again leading to a contradiction.
Hence $\mathrm{supp}(a)$ has size $1$ and $a\in X$. As $r_a=3$, we even find $a\in F$.

Now consider the case that $\Delta$ is a partial linear space with three points per line and at least four lines per point.
Again as all edges have labels different from $0$ and $1$, we find $f$ and hence $a$ to be a primitive idempotent. 
But then \cref{components} implies that $\mathrm{supp}(a)$ is connected. 
Since $r_a\leq 6\leq g-2$, we also find, using \cref{treeindep,specialtree}, that the induced graph on $\mathrm{supp}(a)$ is a tree of diameter at most $4$. 
If $d=g-2$, then $d=4$ and \cref{leaves}(ii) implies   that $r_a$ is at least $d+3+3-4=d+2=6$, contradicting $r_a=4$.
So $d\leq 3\leq g-3$, and \cref{leaves}(i) implies that, when $\mathrm{supp}(a)$ has $n$ vertices and $\ell$ leaves, we have 
$4=r_a\geq n-\ell+1+2\ell=n+\ell+1$.
But, in a tree of diameter $d$ we have $n+\ell+1\geq (d+1)+2+1=d+4$, implying $d=0$.
Hence $\mathrm{supp}(a)$ has size $1$ and $a\in X$. As $r_a=4$, we even find $a\in F$.

This implies that, in case $\Delta$ is a graph and in case it is a partial linear space with three points per line as well, that $\mathrm{Aut}(A)$ fixes the set $F$.

Now consider an element $y\in Y$ and set $b=y^h$.
Fix two edges, (or lines)  $e_0$ and $k_0$ on $y$ and set $e=e_0^h$ and $k=k_0^h$.
Then $y$ is contained in the intersection of the images of $L_{e_0}$ and $L_{k_0}$.
So, $b$ is contained in the intersection of the images of $L_{e}$ and $L_k$.
As $e\neq k$ are in $F$,  we find elements $(x_e),y_e,z_e\in Y$ with $e=\{(x_e),y_e,z_e\}$ and 
$( x_k), y_k,z_k\in Y$ with $k=\{(x_k),y_k,z_k\}$. (Here $x_e$ and $x_k$ are only present in case $\Delta$ is a partial linear space.)
Moreover, 
the images  of $L_{e}$ and $L_k$
are $\langle e,(x_e),y_e,z_e\rangle$ and $\langle k,(x_e), y_k,z_k\rangle$, respectively. 
As the elements of $X=Y\cup F$ are linearly independent, we find the intersection of these subspaces to be
$\langle x\rangle$ where $x$ is one of the vertices (or points) on $e$ or $k$ and $e$ and $k$ intersect in this unique vertex $x$.
As $b\in\langle x\rangle$ is an idempotent, we even find $b=x\in X$.

This all implies that $\mathrm{Aut}(A)$ not only fixes $F$ but also $Y$ and hence $X$ (as a set).
Moreover, we find, as $X$ is a basis for $A$,  that for each $h\in \mathrm{Aut}(A)$ we have $h(X)=X$.
Finally, as two vertices $x\neq y$ in $X$ are on an edge with label $\alpha$ if and only if their product is nonzero and equal to $\alpha(x+y)$, 
$\mathrm{Aut}(A)$ embeds injectively into $\mathrm{Aut}(\Gamma)$, the automorphism group of the labeled graph $\Gamma$.
We find $\mathrm{Aut}(A)$  and $\mathrm{Aut}(\Gamma)$ to be isomorphic.
\end{proof}

If the field $\mathbb{F}$ has size at least $3$, then of course we can label all edges of a graph with elements different from $0$ and $1$.
However, if $\mathbb{F}$ is the field with two elements, this is not possible.
In this case all edges need to be labeled with $1$.
We can still prove a result similar to the above.

\begin{theorem}\label{autgroupthm3}
Suppose $\Delta=(Y,F)$ is a connected graph with each vertex of degree at least $3$.
Let $\Gamma$ be its directed incidence graph.
Suppose all edges of  $\Gamma$ are labeled by $1\in \mathbb{F}_2$. 
Then the automorphism group of the $\mathbb{F}_2$ algebra $A=A_\Gamma$  is isomorphic to $\mathrm{Aut}(\Gamma)$ the automorphism group of the  graph $\Gamma$.
\end{theorem}

\begin{proof}
We follow the proof of the above theorem.
Let $X=Y\cup F$ be the vertex set of $\Gamma$ and fix an element $f\in F$ and $h\in \mathrm{Aut}(A)$.
Set $a=f^h$.
The rank of  $L_f$ and $R_f$ is three. So, also $L_a$ and $R_a$ have rank $3$.
We will show $a\in F$.
The difference with the proof of \cref{autgroupthm2} is that $f$ and then also $a$ has eigenvalue $1$ with (algebraic) multiplicity $3$, but only a $2$-dimensional eigenspace. 

By \cref{components} we find that the support of $a$ has at most three connected components.
If it has three components, then $a$ would be semi-simple. So, there are at most two components.
Suppose $\mathrm{supp}(a)$ does have two components. By \cref{tree} these components  are trees  of diameter at most $2$.
Then write $a=a_1+a_2$ with $\mathrm{supp}(a_1)$ and $\mathrm{supp}(a_2)$ these two components which we assume
to be of diameter $d_1$ and $d_2$, respectively.

If $d_i>1$, then, by \cref{treebound}, the image of the subspace spanned by $\mathrm{supp}(a_i)$ under $L_a$ has dimension at least $d_i\geq 2$
and is contained in this subspace.
If $d_i=1$, this is also true, as then $a_i=y+z$ for some adjacent $y,z\in X$, and we find  $ay=z$ and  $az=y$ in the image.
So, if  $d_1,d_2\geq 1$, we find $r_a>3$, which is against our assumption.
Hence, we can assume $a_2$ to be in $X$.
If $d_1>0$,  then by the above we find that the  image of $L_a$ is contained in the span of $a_2$ and $\mathrm{supp}(a_1)$ and $d_1\leq 2$.

Each of the leaves of $\mathrm{supp}(a_1)$ has at least one neighbor $y$ outside $\mathrm{supp}(a_1)$.
If this neighbor is not adjacent to $a_2$, then we find $a_1y$ to be a nonzero vector in the image of $L_a$ outside 
the space spanned by $\mathrm{supp}(a)$.
This can not happen. Hence these neighbors are adjacent to $a_2$. As the girth is at least $6$, each leaf of  $\mathrm{supp}(a_1)$ can have only one neighbor outside
$\mathrm{supp}(a)$.
So leaves of $\mathrm{supp}(a_1)$ and then also $a_2$ (which is at distance $2$ from such leaf) have just two neighbors.
Thus, these leaves and $a_2$ are elements of $F$ and therefore  are not adjacent in $\Gamma$.
This implies that  $\mathrm{supp}(a_1)$ is a path with $3$ vertices.

Let $u$ be the unique vertex of $\mathrm{supp}(a_1)$ which is not a leaf.
Then by the above it is at distance $3$ from $a_2$. So, it is in $Y$ and  has at least three neighbors.
But then  one of them, say $w$, is outside $\mathrm{supp}(a_1)$, and $aw$ is in the image of $L_a$ but outside the span of $\mathrm{supp}(a_1)$, leading to a contradiction. 
Hence, $d_1=0$ and we also find $a_1\in X$.

But then both $a_1$ and $a_2$ are in $X$ and  have at most one common neighbor. So we can find different neighbors $z_1$ and $z_2$ of  $a_1$ and $a_2$, respectively,
with $z_1$ not adjacent to $a_2$ and $z_2$ not to $a_1$.
But then we find the elements $aa_1=a_1,aa_2=a_2$ as well as $az_1=a_1+z_1$ and $az_2=a_2+z_2$ in the image of  $L_a$, a contradiction with $r_a=3$.
We  conclude that $\mathrm{supp}(a)$ is connected.

The graph $\Gamma$ has girth at least $6$. So the rank $r_x$ of $L_x$ satisfies $r_x\leq g-3$ and we can conclude by \cref{tree}
that $\mathrm{supp}(a)$ is a tree. Now we can finish the proof just as in the proof of \cref{autgroupthm2}.
Indeed, as above, we can conclude $a\in X$, and, more precisely, $a\in F$.
Hence $\mathrm{Aut}(A)$ fixes the set $F$, and  then also $X$, proving $\mathrm{Aut}(A)$  and $\mathrm{Aut}(\Gamma)$ to be isomorphic.
\end{proof}

The above results enables us to construct (axial) algebras with automorphism group isomorphic to the automorphism group of various graphs or  geometries
related to (spherical) buildings and even sporadic simple groups.

We start with some examples related to incidence graphs of graphs.
(We refer the reader to \cite{drg,BC} for more information on the graphs under investigation.)

\begin{example}
In the above results we do not require the graphs (or partial linear spaces) $\Delta$ to be finite.
Suppose $\Pi=(P,L)$ is a nondegenerate polar space with at least three points per line. 
Then $\Pi$ can be recovered from its collinearity graph $\Delta$. (Indeed, a line on two points $x,y\in P$ equals $\{x,y\}^{\perp\perp}$,
where $\perp$ denotes the relation on $P$ of being equal or collinear.)
Now let $\Gamma$ be the directed incidence graph of $\Delta$, and assign to each edge of $\Gamma$ a fixed label $\alpha\in \mathbb{F}$ different from $0$ and $1$. 
Then $A_\Gamma$ is a simple commutative axial algebra with $$\mathrm{Aut}(A_\Gamma)\simeq\mathrm{Aut}(\Gamma)\simeq\mathrm{Aut}(\Delta)\simeq \mathrm{Aut}(\Pi).$$
If the field $\mathbb{F}$ contains at least $4$ elements we can fix distinct $\alpha$ and $\beta$ different from $0$ and $1$
and label all edges $(p,\ell)$ of $\Gamma$, where $p$ is a point and $\ell$ a line containing $p$ with $\alpha$, and all edges $(\ell,p)$ with $\beta$.
Then $A_\Gamma$ is a simple non-commutative axial algebra with $\mathrm{Aut}(A_\Gamma)$ isomorphic to $\mathrm{Aut}(\Pi)$. 
\end{example}

\begin{example}
Various of the sporadic simple groups are constructed as automorphism groups of graphs.
Using the directed  incidence graphs of these graphs and labeling them with a constant $\alpha$ different from $0,1$ will provide us with
simple commutative  axial algebras admitting the simple group as (sub)group of automorphisms.

If we label edges of the form $(v,e)$ where $v$ is a vertex and $e$ an edge with $\alpha$ and edges of the form $(e,v)$ with a $\beta$ different $0,1$ and $\alpha$,
then we find simple non-commutative  axial algebras admitting the simple group as (sub)group of automorphisms. For example:

The Higman-Sims group $\mathrm{HS}$ is constructed as a subgroup of the automorphism group of the Higman-Sims graph, a graph 
with $100$ vertices, and valency $22$. This graph has $1100$ edges and the labeled directed incidence graph is then giving rise to
$1200$-dimensional algebras whose automorphism group is isomorphic to $\mathrm{HS}:2$.

The Hall-Janko group $\mathrm{HJ}$ is also an automorphism group of a graph with $100$ vertices, but now of valency $36$ and hence with $1800$ edges.
Labeling the incidence graph gives rise to algebras of dimension $1900$ with automorphism group  $\mathrm{HJ}:2$.

The McLaughlin graph on $275$ vertices with valency $112$  and $15400$ edges admits the automorphism group of the 
sporadic McLaughlin group $\mathrm{McL}$ as group of automorphisms, yielding $15675$-dimensional algebras with automorphism group
$\mathrm{McL}:2$.
\end{example}

Next we provide some examples related to partial linear spaces with three points per line.
For a description of these geometries we also refer to \cite{drg} and \cite{BC}.
\begin{example}
Consider $\Delta$ to be a root group geometry defined over the field $\mathbb{F}_2$ (i.e. with three points per line).
Then with $\Gamma$ being the incidence graph of the geometry labeled with a fixed $\alpha\in \mathbb{F}$ different from $0,1$ (so assuming $\mathbb{F}$ is of size at least three), we find algebras $A_\Gamma$ with automorphism group 
being the automorphism group of $\Delta$, which is the automorphism group of a group of Lie type over $\mathbb{F}_2$, provided each point is on at least four lines.
\end{example}

\begin{example}
If we take for $\Delta$ the regular near hexagon on $729$ points obtained from the extended ternary Golay code with automorphism group 
$G=2^6.2.\mathrm{M}_{12}$, then $\Delta$ has three points per line and $12$ lines per point.
So the incidence graph $\Gamma$ is a graph with $5\cdot 729$ vertices.
Hence labeling its edges with an $\alpha\in \mathbb{F}$ different from $0$ and $1$ gives rise to an algebra of dimension
dimension $3645$ and automorphism group $G$.


The near hexagon on $759$ points related to the Mathieu group $\mathrm{M}_{24}$ has $5\cdot 759$ lines of size $3$.
Labeling the edges of the incidence graph provides an algebra of $A$ of dimension $6\cdot 759=4554$ with automorphism group  $\mathrm{M}_{24}$.


Let $\Delta$ be the near-octagon related to the Hall-Janko group $\mathrm{HJ}$. This near-octagon has $315$ points, $3$ points per line and $5$ lines per point. The number of lines is $525$.
Then its incidence graph $\Gamma$ is then a graph with $315+525=840$ vertices.
Labeling $\Gamma$ with a fixed $\alpha\neq 0$ provides us with an $840$-dimensional algebra $A_\Gamma$ with automorphism group $\mathrm{HJ}:2$.

The algebras thus obtained are all commutative axial algebras.
Labeling the edges $(p,\ell)$ where $p$ is a point and $\ell$ a line with $\alpha$ and the reversed edges $(\ell,p)$ with a $\beta$ different from $0,1,\alpha$ yields
simple non-commutative axial algebras with the corresponding almost simple groups as automorphism group.
\end{example}

\begin{example}
Let $\Delta$ be a Fischer space related to one of the three sporadic Fischer groups. 
Then with $\Gamma$ its incidence graph labeled with a nonzero $\alpha$ (or distinct $\alpha$ and $\beta$) different from $0,1$, we obtain simple
(non-) commutative axial algebras $A_\Gamma$ with automorphism group 
the automorphism group of these sporadic Fischer groups. 

We can generalize this as follows. Let $D$ be a conjugacy class of involutions in a simple group $G$ and suppose there are at least two elements
$d,e\in D$ with the order of $de$ equal to $3$.
Then let $L$ be the set of triples $\{d,e,f\}$ in $D$ with $de$ of order $3$ and $f=d^e=e^d$.
Then $\Delta=(D,L)$ is a partial linear space with three points per line.
As $G$ is assumed to be simple, $\Delta$ will be connected.
Taking $\Gamma$ to be the incidence graph of $\Delta$ in which all edges are labeled by a nonzero element provides an algebra
$A_\Gamma$ with $G$ contained in the automorphism group.
The full automorphism group of $\Delta$ and $\Gamma$ and hence also $A_\Gamma$  then contains $\mathrm{Aut}(G)$.

In particular, we can take for $G$ the Monster or the Baby Monster and $D$ the class of $2B$ involutions to obtain
algebras with these groups contained in the automorphism group.
The full automorphism group of $A_\Gamma$ is isomorphic to $\Aut(\Delta)$, which certainly contains $\Aut(G)$, but may be bigger.
\end{example}

\begin{remark}
We finish this section with the following observation.
Let $\mathcal{F}$ be a set containing $0$ and $1$.
Suppose $\Gamma$ and $\Gamma'$ are graphs labeled with elements from a set $\mathcal{F}$, different from $0$ and $1$.
If both $\Gamma$ and $\Gamma'$ satisfy the conditions of one of the above theorems, but are non-isomorphic as labeled graphs,
then the algebras $A_\Gamma$ and $A_{\Gamma'}$ are non-isomorphic.
Indeed, any isomorphism between the two algebras would induce an isomorphism between the two graphs as follows by  the above \cref{autgroupthm2,autgroupthm3}.
\end{remark}

\section{Every group is automorphism group of a  simple (axial) algebra}
\label{section:frucht}

In \cite{popov} Popov raised the question  whether every finite group is isomorphic to  the full automorphism group of a finite dimensional simple algebra. He, together with Gordeev, provided a positive answer to this question. See \cite{gordeev}.
Actually they prove for any field $\mathbb{F}$ containing sufficiently many elements and its algebraic closure $\mathbb{K}$, that 
for every
linear algebraic $\mathbb{F}$-group $G$ there exists a finite dimensional simple $\mathbb{F}$-algebra $A$ such that the algebraic group $\mathrm{Aut}(A\otimes_\mathbb{F} \mathbb{K})$ is defined over $\mathbb{F}$ and $\mathbb{F}$-isomorphic to $G$.
As a corollary they find that for every finite group $G$ there is a finite dimensional algebra $A$ over $\mathbb{F}$ with $\mathrm{Aut}(A)$ isomorphic to $G$.

Costoya, Mu\~{n}oz, Tocino, and Viruel \cite{evolution} construct for each finite group $G$ a simple evolution algebra over a field $\mathbb{F}$
of size at least $2\cdot|G|$ with automorphism group isomorphic to $G$, and thus also provide an affirmative answer to
Popov's question. 

Now \cref{autgroupthm2} and \cref{autgroupthm3}, 
together with  classical results of   Frucht \cite{frucht39,frucht49}, de Groot \cite{groot} and Sabidussi \cite{sabidussi}, 
provide for each group $G$ and  field $\mathbb{F}$ a construction of an infinite series of  algebras over $\mathbb{F}$ with automorphism group isomorphic to $G$.
For a given set $\mathcal{F}\subset \mathbb{F}$ containing $0,1$ and at least one (or two) other element, we can even make sure that these  algebras are (non-commutative) axial algebras with a set of axes satisfying  the graph fusion laws $\mathcal{G}(\mathcal{F})$.

For finite groups these graphs can be constructed such that each vertex has degree at least $3$, see \cite{sabidussi},
and we can use \cref{autgroupthm2} to find algebras with $G$ as its automorphism group.

For infinite groups, however, to be able to apply \cref{autgroupthm2} we still have to make sure that in the resulting graphs of 
(our modified versions of) the constructions of Frucht \cite{frucht39,frucht49}, de Groot \cite{groot} and Sabidussi \cite{sabidussi, sabidussi2} all vertices have degree at least $3$, so that we can apply \cref{autgroupthm2} to find algebras with $G$ as its automorphism group.

The constructions of Frucht, de Groot and Sabidussi all use the following approach. 
They first consider a set $S$ of generators for $G$ and the Cayley graph of $G$ with respect to the generating set $S$.
All edges with the same generator as label are then replaced by a graph with trivial automorphism group.
So, if $G$ is an infinite group and   we start with a set $S$ of at least four generators, 
and replacing edges by graphs with trivial automorphism group in which each vertex has degree at least three, then we end up
with a graph in which all vertices have degree at least $3$. Edges with different labels should be replaced by non-isomorphic graphs. 
So, it only needs to be checked that there are enough of such graphs. 
So $|S|$ non-isomorphic graphs with trivial automorphism group in which each vertex
has valency at least three will suffice. (Here $|S|$ denotes a cardinal number.)

We use a modified version of de Groot's approach \cite{groot} for constructing such graphs.
For each $s\in S$ let $\Theta_s$ be a graph containing as a subgraph the Frucht graph $\Phi$, which is a graph with $12$ vertices in which each vertex has degree $3$
and which has only the identity map as automorphism.
Furthermore for fixed adjacent vertices $q$ and $r$ of $\Phi$ inside $\Theta_s$ attach to $r$ a tree $T_s$ rooted at $r$ of countable depth in which the degrees of the vertices are cardinals bigger than $|S|$, 
such that no two vertices have the same degree. See \cref{graphfigure}. Clearly we can construct at least $|S|$ such graphs, with $\Theta_s$ isomorphic to $\Theta_{s'}$ if and only if $s=s'$ (by varying the valency of $r$). 
For each $s$ in $S$
replace all edges $(a,b=as)$ with label $s$ in the Cayley graph of $G$ with generating set $S$ by the graph $\Theta_s$, and additional edges $\{a,q\}$ and $\{r,b\}$.

\begin{figure}[h]
\begin{tikzpicture}
 
 \draw[color=black,dashed]
(0,0) ellipse [x radius=1.1, y radius=1.2]
(1.15,1.75) ellipse [x radius=0.5, y radius=2.3];
\filldraw
(-2,0) circle [radius=2pt]
(-1,0) circle [radius=2pt]
(1,0) circle [radius=2pt]
(2,0) circle [radius=2pt]
(-0.5,0.5) circle [radius=2pt]
(-0.5,-0.5) circle [radius=2pt]
(.5,0.5) circle [radius=2pt]
(.5,-0.5) circle [radius=2pt]
(1,1) circle [radius=2pt]
(1.2,1) circle [radius=2pt]
(1.4,1) circle [radius=2pt]
(0.8,1) circle [radius=2pt]
(-4,3) circle [radius=2pt]
(-2,3) circle [radius=2pt]

;

\draw (-1,0)--(-0.3,0.7);
\draw (-0.5,0.5)--(-0.3,0.3);
\draw (-2,0)--(2,0);
\draw (-1,0)--(-0.3,-0.7);
\draw (-0.5,-0.5)--(-0.3,-0.3);

\draw (1,0)--(.3,0.7);
\draw (0.5,0.5)--(0.3,0.3);
\draw (1,0)--(.5,0);
\draw (1,0)--(.3,-0.7);
\draw (0.5,-0.5)--(0.3,-0.3);

\draw (1,0)--(1,1);
\draw (1,0)--(1.2,1);
\draw (1,0)--(1.4,1);
\draw (1,0)--(0.8,1);

\draw (-4,3)--(-2,3);
\draw[->] (-3,2)--(-2,1);

\draw (-4,2.7) node {$a$};
\draw (-2,2.7) node {$b$};
\draw (-3,3.3) node {$s$};

\draw (-2,-0.3) node {$a$};
\draw (-1,-0.3) node {$q$};
\draw (-0.2,1) node {$\Theta_s$};
\draw (1.2,1.5) node {$T_s$};

\draw (1,-0.3) node {$r$};
\draw (2.2,-0.3) node {$b$};

\end{tikzpicture}
\caption{Replace the edge $(a,b=as)$ by a subgraph.}\label{graphfigure}
\end{figure}
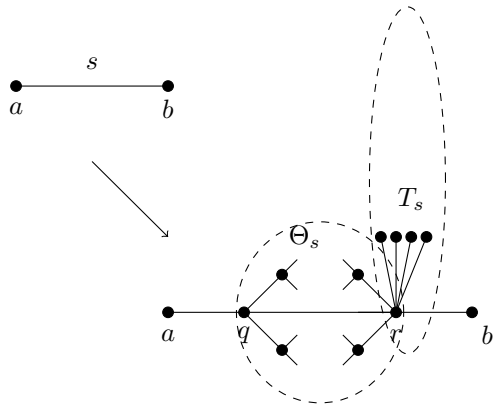

In the resulting graph, all vertices from the original Cayley graph have valency $|S|> 3$, all vertices of the subgraphs $\Phi$
different from $q$ and $r$ have valency $3$, $q$ has valency $4$ and vertices in the tree on $r$ have valency bigger than $|S|$. 
From this resulting graph we can recover the original Cayley graph.
Indeed, we can recover the vertices of $\Theta_s$ different from $p,q$ as vertices of degree $3$, and the vertices $q$ and  $r$
as being adjacent to these vertices and having valency $4$ and $>|S|$, respectively.
This means we can recover the directed edge $(a,b)$ as well as the labeling by the isomorphism type of $T_s$ and hence the original Cayley graph.
As in \cite{groot} we can conclude that the resulting graph, denoted by $\Delta$, is an undirected graph  with automorphism group isomorphic to $G$ in which each vertex has degree at least $3$.

We notice that we can  construct infinitely many non-isomorphic graphs $\Delta$, by using different graphs with trivial automorphism group
and constant valency $3$ instead of the Frucht graph $\Phi$ or by varying the valency of $r$.  
So, the above implies the following result.

\begin{theorem}
Let $G$ be a group. Then there exist infinitely many non-isomorphic ordinary graphs $\Delta$ with
each vertex of valency at least $3$ and $\mathrm{Aut}(\Delta)$ isomorphic to $G$.
\end{theorem}

Now for a group $G$ let $\Delta$ be a graph with all vertices of valency at least $3$ and $\mathrm{Aut}(\Delta)$ isomorphic to
$G$, and take $\Gamma=(X,E)$ to be the directed incidence graph of $\Delta$. Then $G$ is also the automorphism group of $\Gamma$.
If $\mathbb{F}$ is the field with two elements, we label all edges of $\Gamma$ with $1$.
If $\mathbb{F}$ has at least three (or four) elements and we
can label the edges with elements different from $0$ and $1$, such that two edges in a single $\mathrm{Aut}(\Gamma)$-orbit have the same value. 
The group $G$ is then also the automorphism group of the so obtained labeled graph, also denoted by $\Gamma$.
We notice that edges $(v,e)$ and $(e,v)$ of $\Gamma$
where $v$ is a vertex on an edge $e$ of $\Delta$ are in distinct orbits. So, we can label them by the same element of $\mathcal{F}$ to obtain
commutative algebras $A_{\Gamma}$, or by different elements from $\mathcal{F}$ (if possible) to obtain non-commutative algebras $A_{\Gamma}$.

Applying \cref{autgroupthm2} (or \cref{autgroupthm3} in case $\mathbb{F}$ has just two elements),
we obtain that $G$ is isomorphic to the automorphism group of the algebra $A_{\Gamma}$,
which is simple, as follows from the observation that  $\Gamma$ does not contain cliques of size $3$, see \cref{simplecor}.
Moreover, in case $\mathcal{F}$ has more than two elements,  $A_\Gamma$ is an axial algebras with a set of primitive axes $X$ 
satisfying the fusion law $\mathcal{G}(\mathcal{F})$, where $\mathcal{F}$ contains $\{0,1\}$ and  all used labels.

Hence, we have obtained the following result which provides for every group $G$ and  field $\mathbb{F}$ infinitely many non-isomorphic simple algebras with automorphism group 
isomorphic to $G$.

\begin{theorem}\label{fruchttheorem}
Let $G$ be a  group and $\mathbb{F}$ a field.

\begin{enumerate}
\item If $\mathbb{F}$ contains at least three elements, then
there exist infinitely many non-isomorphic  directed graphs $\Gamma=(X,E)$ with edges labeled by elements of $\mathbb{F}\setminus\{0,1\}$
such that $A_\Gamma$ is a  commutative simple axial  algebra over $\mathbb{F}$ generated by a set of primitive idempotents $X$
of graph type $\mathcal{G}(\mathcal{F})$ for  $\mathcal{F}$ being the set containing $0,1$ and all labels, and $G$ isomorphic to $\mathrm{Aut}(A_\Gamma)$.
\item
If $\mathcal{F}$ contains at least four elements, then there exist infinitely many non-isomorphic  directed graphs $\Gamma=(X,E)$ with edges labeled by elements of $\mathcal{F}\setminus \{0,1\}$
such that $A_\Gamma$ is a  non-commutative simple  axial algebra over $\mathbb{F}$ generated by a set of primitive idempotents $X$
of graph type $\mathcal{G}(\mathcal{F})$ for  $\mathcal{F}$ being the set containing $0,1$ and all labels, and $G$ isomorphic to $\mathrm{Aut}(A_\Gamma)$.
\item
If $\mathbb{F}$ has only two elements, then there exists infinitely many non-isomorphic  directed graphs $\Gamma$ with edges labeled with $1$ such that $A_\Gamma$ over $\mathbb{F}_2$ is a  simple 
algebra and has automorphism group isomorphic to $G$.
\end{enumerate}

Moreover, in case $G$ is a finite group, the graphs $\Gamma$ in the above statements can be
assumed to be finite, implying the algebras $A_\Gamma$ to be finite dimensional.
\end{theorem}

\cref{fruchttheorem} implies \cref{groupthm}.

For finite groups $G$ we can view the above theorem as the counterpart to the following result
due to  Gorshkov, M$^\mathrm{c}$Inroy, Shumba and Shpectorov.

\begin{theorem}\cite{gorshkov2023automorphism}
Let $A$ be a finite-dimensional commutative axial algebra over a field $\mathbb{F}$ of characteristic
not two with fusion law $\star$ on the set $\mathcal{F}\subseteq \mathbb{F}$. If $\frac{1}{2}\not\in \mathcal{F}$, then $\mathrm{Aut}(A)$ is finite.
\end{theorem}

\begin{remark}
We can not only use the directed incidence graph of graphs to construct an algebra for every finite group $G$ as in the above theorem, 
but also  the directed incidence graph of some finite partial linear spaces with three points per line.
Mendelsohn \cite{mendelsohn} has shown that for each finite group $G$ there does exist a finite Steiner triple system $\Delta$ with at least $4$ lines per point and $G$ being isomorphic to the automorphism group of $\Delta$.
By taking $\Gamma$ to be the directed  incidence graph of $\Delta$ and labeling its edges with an element different from  $0$ and $1$ in the field $\mathbb{F}$ (different from $\mathbb{F}_2$)  we obtain, by  \cref{autgroupthm2} or \cref{autgroupthm3}, a finite dimensional commutative and simple axial algebra $A_\Gamma$
with $G$ as automorphism group.
If the field contains at least $4$ elements, we can label the edges in such a way that we obtain a finite dimensional non-commutative and simple axial algebra
$A_\Gamma$ with automorphism group isomorphic to $G$.
\end{remark}

\bigskip

\bibliographystyle{plain}

\bibliography{frucht.bib}

\vspace{2cm}

\parindent=0pt
Hans Cuypers\\
Department of Mathematics and Computer Science \\
Eindhoven University of Technology\\
P.O. Box 513 5600 MB, Eindhoven\\
The Netherlands\\
email: f.g.m.t.cuypers@tue.nl

\end{document}